\documentclass[reqno, 12pt,a4letter]{amsart}
\usepackage{amsmath,amsxtra, amssymb, latexsym, amscd,amsthm}
\usepackage{graphicx,color}
\usepackage[active]{srcltx}
\usepackage[utf8]{inputenc}
\usepackage[mathscr]{euscript}
\usepackage{mathrsfs}
\usepackage[english]{babel}
\usepackage{color}
\newtheorem {theorem}{Theorem}[section]

\newtheorem {lemma}{Lemma}[section]
\newtheorem {example}{Example}[section]
\newtheorem {definition}{Definition}[section]

\def\ar{a\kern-.370em\raise.16ex\hbox{\char95\kern-0.53ex\char'47}\kern.05em}
\def\ees{{\accent"5E e}\kern-.385em\raise.2ex\hbox{\char'23}\kern-.08em}
\def\eex{{\accent"5E e}\kern-.470em\raise.3ex\hbox{\char'176}}
\def\AR{A\kern-.46em\raise.80ex\hbox{\char95\kern-0.53ex\char'47}\kern.13em}
\def\EES{{\accent"5E E}\kern-.5em\raise.8ex\hbox{\char'23 }}
\def\EEX{{\accent"5E E}\kern-.60em\raise.9ex\hbox{\char'176}\kern.1em}
\def\ow{o\kern-.42em\raise.82ex\hbox{
        \vrule width .12em height .0ex depth .075ex \kern-0.16em \char'56}\kern-.07em}
\def\OW{O\kern-.460em\raise1.36ex\hbox{
        \vrule width .13em height .0ex depth .075ex \kern-0.16em \char'56}\kern-.07em}
\def\UW{U\kern-.42em\raise1.36ex\hbox{
        \vrule width .13em height .0ex depth .075ex \kern-0.16em \char'56}\kern-.07em}
\def\DD{D\kern-.7em\raise0.4ex\hbox{\char '55}\kern.33em}

\title[]{Second-order KKT optimality conditions for multi-objective optimal  control problems}

\author{BUI TRONG KIEN$^{1,2}$}
\address{$^1$Institute of Research and Development, Duy Tan University, 03 Quang Trung, Da Nang, Vietnam }
\address{$^2$Department of Optimization and Control Theory, Institute of Mathematics, VAST, 18 Hoang Quoc Viet road,  Hanoi, Vietnam}
\email{btkien@math.ac.vn}

\author{NGUYEN VAN TUYEN$^3$}
\address{$^3$Department of Mathematics, Hanoi Pedagogical University 2, Xuan Hoa, Phuc Yen, Vinh Phuc, Vietnam}

\email{tuyensp2@yahoo.com; nguyenvantuyen83@hpu2.edu.vn}

\author{Jen-Chih Yao$^4$}
\address{$^4$Center for General Education, China Medical University, Taichung 40402, Taiwan}
\email{yaojc@mail.cmu.edu.tw}

\keywords{Second-order optimality conditions~$\cdot$~Multi-objective optimal  control problems~$\cdot$~Weak Pareto solutions~$\cdot$~The Robinson constraint qualification}

\subjclass{49K15~$\cdot$~90C29}

\date{ \today}

\begin{document}
\maketitle

\begin{abstract} In this paper, we study  second-order necessary and sufficient optimality conditions of Karush--Kuhn--Tucker-type
 for locally optimal solutions in the sense of Pareto to a class of multi-objective optimal control problems with mixed pointwise constraint.
To deal with the    problems, we first derive second-order
optimality conditions for    abstract multi-objective optimal
control problems which satisfy the    Robinson constraint
qualification. We then apply the obtained  results to our concrete
problems. The proofs of obtained results are direct, self-contained
without using scalarization techniques.
\end{abstract}

\section{Introduction}\label{Introduction}
Let $L_j\colon [0, 1]\times\mathbb{R}^n\times\mathbb{R}^l\to \mathbb{R}$
with $j=1, 2, \ldots, m$,  $\varphi\colon [0,
1]\times\mathbb{R}^n\times\mathbb{R}^l\to \mathbb{R}^n$, and $g \colon [0,
1]\times\mathbb{R}^n\times\mathbb{R}^l\to \mathbb{R}$ be given
functions. We consider the  multi-objective optimal control problem
of finding a control vector $u\in L^\infty([0, 1], \mathbb{R}^l)$
and the corresponding state $x\in C([0, 1], \mathbb{R}^n)$ which solve
\begin{align}
&\textrm{Min}_{\,\mathbb{R}^m_+}\ I(x, u)\label{P1}\\
&\textrm{s.t.} \quad \;  x(t)=x_0+\int_0^t\varphi(s, x(s),
u(s))ds\ \ \text{for a.e.}\ \
t\in [0, 1],\label{P2}\\
& \quad\quad\ \ g(t, x(t), u(t))\leq 0\ \ \text{for a.e.}\ \ t\in [0,
1].\label{P4}
\end{align}
Here $x_0$ is a given vector in $\mathbb{R}^n$ and the multi-objective function $I$ is given by
$$I(x, u)=(I_1(x, u), I_2(x, u), \ldots, I_m(x, u)),
$$
where
\begin{equation*}
I_j(x, u):=\int_0^1 L_j(t, x(t), u(t))dt.
\end{equation*}
We denote by (MCP) the problem
\eqref{P1}--\eqref{P4} and by $\Phi$ its feasible set, that is,
$\Phi$ consists of couples $(x, u)\in C([0, 1], \mathbb{R}^n)\times
L^\infty([0, 1], \mathbb{R}^l)$ which satisfy constraints
\eqref{P2}--\eqref{P4}.

The multi-objective optimal control problems are important in
mechanics and economy. For example, when we want to minimize
energy and time of a system, we need to use two-objective optimal
control which has a form like \eqref{P1}--\eqref{P4} (see for
instance \cite{Kaya}). Recently, problem (MCP) has been
studied by several mathematicians. For papers which have a close
connection to the present work,  we refer the readers to
\cite{Bellaassali,Bonnel2,Greksch,Kaya,Kien1,Ngo,Olive,Olive2,Shao,Zhu}
and the references therein. In these papers, the authors mainly studied
numerical methods and first-order necessary optimality conditions
for multi-objective optimal control problems. However, to the best
of our knowledge, so far there have been no  papers  investigating
second-order optimality conditions for multi-objective optimal
control problems. The study of second-order optimality conditions
for optimization problems as well as for multi-objective optimal
control problems is a fundamental topic in optimization theory. The
second-order optimality conditions play an important role in
solution stability and numerical methods of finding optimal
solutions.

In this paper, we will focus on  deriving second-order necessary
optimality conditions  and second-order  sufficient optimality
conditions of Karush--Kuhn--Tucker (KKT) type for the multi-objective
optimal control problem (MCP). In order to establish second-order
KKT optimality conditions for the (MCP), we first derive second-order
optimality conditions for abstract multi-objective optimal control
problems which satisfy the Robinson constraint qualification. We
then apply the obtained results to our concrete problem.

In contrast with multi-objective optimal control problems, there
have been some papers dealing with second-order KKT optimality
conditions for vector optimization problems recently. For papers of
this topic, we refer the reader to \cite{Gfrerer,Jimenez03,Jimenez04,Ning}  and  references given therein.
In \cite{Gfrerer,Ning}, the second-order KKT optimality conditions were derived by scalarization method via
 the so-called oriented distance function which was used by Ginchev et al. in \cite{Ginchev} for the first time.
 However, this approach has also some certain limits because the oriented distance function is often nonsmooth.
 In \cite{Jimenez03,Jimenez04}, by using Motzkin's theorem of the alternative, Jim\'enez et al. presented some
 second-order KKT optimality conditions for vector optimization problems  under  suitable constraint qualification
 conditions. Although the constraint qualification conditions used in \cite{Jimenez03,Jimenez04}  are  weaker than
 the  Robinson constraint qualification, the ``sigma''  terms in the obtained second-order conditions do not vanish.
 In addition, those results  can not apply to the (MCP) directly as the Robinson constraint qualification does not hold for the (MCP).

In the present paper, we derive second-order KKT optimality conditions for vector optimization directly via separation theorems.  We then establish second-order KKT conditions without sigma terms for an abstract multi-objective optimal control problem under the Robinson constraint qualification. It is worth pointing out that our method  is natural and intrinsic. The obtained results approach a theory of no-gap second-order optimality conditions for multi-objective optimal control problems.

The paper is organized as follows. In Section \ref{main-results}, we  set up notation and terminology, and state main results. Section \ref{abstract-MP} is intended to derive second-order KKT optimality conditions for a class of vector
optimization problems. In Section \ref{abstract-MCP}, we establish second-order KKT necessary optimality conditions for an abstract multi-objective
optimal control problem, which is based on the obtained result of
Section \ref{abstract-MP}. The proofs of the main results will be provided in Section \ref{proofs}. In Section \ref{illustrate-example}, we give  some examples to illustrate the main results.

\section{Assumptions and statements of the main results}
\label{main-results}
In this section, $C([0, 1], \mathbb{R}^n)$ is the Banach space of
continuous vector-valued functions $x \colon [0, 1]\to\mathbb{R}^n$ with the norm $\|x\|_0=\displaystyle\max_{t\in[0, 1]}|x(t)|$ and $\mathbb{R}^n$  is the Euclidean space of $n$-tuples $\xi=(\xi_1, \ldots, \xi_n)$ with the norm $|\xi|=\left(\sum_{i=1}^{n}\xi_i^2\right)^{\frac{1}{2}}$.
 For each $1\leq p\leq\infty$, $L^p([0, 1], \mathbb{R}^l)$ stands for the Lebesgue
spaces with the norms $\|\cdot\|_p$. For convenience, we put
$X=C([0, 1], \mathbb{R}^n)$ and $U= L^\infty([0, 1], \mathbb{R}^l)$.
In the sequel, $L$ and $\phi $ stand for $(L_1, \ldots, L_m)$ and  $(\varphi, g)$, respectively. Define
$$
Q=\left\{v\in L^\infty\left([0, 1], \mathbb{R}\right)\;|\; v(t)\leq
0, {\rm a.e.}\ t\in [0, 1]\right\}.
$$

Let us impose the following assumptions  on $L$ and $\phi$.
\begin{enumerate}
    \item [$(H1)$] The  function $L$ is a Carath\'{e}odory function
    and $\phi$ is a continuous mapping. For a.e. $t\in [0, 1]$, $L(t,
    \cdot, \cdot)$ is of class $C^2$ while $\phi(t, \cdot, \cdot)$ is of
    class $C^2$ for all $t\in [0, 1]$ . Besides, for each $M>0$,
    there exist numbers $k_{LM}>0$ and $k_{\phi M}>0$ such that
    \begin{align*}
    |L(t, x_1, u_1)-L(t, x_2, u_2)|&+|\nabla L(t, x_1, u_1)-\nabla L(t,
    x_2, u_2)|+\\
    &+|\nabla^2 L(t, x_1, u_1)-\nabla^2 L(t, x_2, u_2)|\leq
    k_{LM}(|x_1-x_2| +|u_1-u_2|)
    \end{align*} for a.e. $t\in [0, 1]$ and for all $x_i\in\mathbb{R}^n$,
    $ u_i\in\mathbb{R}^l$ satisfying $|x_i|\leq M$, $|u_i|\leq M$ with $i=1,2$, and
    \begin{align*}
    |\phi(t, x_1, u_1)-\phi(t, x_2, u_2)|&+|\nabla \phi(t, x_1,
    u_1)-\nabla \phi(t, x_2, u_2)|+\\
    &+\left|\nabla^2 \phi(t, x_1, u_1)-\nabla^2 \phi(t, x_2, u_2)\right|\leq
    k_{\phi M}(|x_1-x_2| +|u_1-u_2|)
    \end{align*} for all $t\in [0, 1]$, $x_i\in\mathbb{R}^n$,
    $ u_i\in\mathbb{R}^l$ satisfying $|x_i|\leq M$, $|u_i|\leq M$ with $i=1,2$.
    Moreover, we require that the functions
    $$
    L(t, 0,0), |\nabla L(t, 0, 0)|, \left|\nabla^2 L(t, 0, 0)\right|, \phi(t, 0,
    0), |\nabla\phi(t, 0, 0)|, \left|\nabla^2\phi(t, 0, 0)\right|
    $$ belong to $L^\infty([0,1], \mathbb{R})$.
    \item [$(H2)$] Given a couple $(\bar x, \bar u)\in \Phi$, there
    exist $i_0\in\{1,2, \ldots, l\}$ and $\alpha>0$ such that
    $$
    |g_{u_{i_0}}(t, \bar x(t), \bar u(t))|\geq \alpha\ \ \text{for a.e.}\ \ t\in [0,
    1].
    $$
\end{enumerate}

Hereafter, $L[t], \varphi[t], g[t], L_x[t], \varphi_u[t], g_u[t]$
and so on,  stand for
\begin{align*}
&L(t, \bar x(t), \bar u(t)),\  \varphi(t, \bar x(t),\ \bar u(t)),
g(t, \bar x(t), \bar u(t)),\\
&L_x(t, \bar x(t), \bar u(t)),\  \varphi_u(t, \bar x(t),\ \bar
u(t)), g_u(t, \bar x(t), \bar u(t)), \ldots
\end{align*}
 We denote by $\mathbb{R}^m_+$ the nonnegative orthant of $\mathbb{R}^m$, where $m\in\mathbb{N}:=\{1, 2, \ldots\}$. The interior of $\mathbb{R}^m_+$ is denoted by $\mathrm{int}\,\mathbb{R}^m_+$.
\begin{definition}{   Assume that $\bar z=(\bar x, \bar u)$ is a feasible point of the (MCP). We say that:
\begin{enumerate}
    \item [(i)] $\bar z$ is a  locally weak Pareto solution  of the (MCP) if there exists $\epsilon>0$ such
    that for all $(x, u)\in (B(\bar x, \epsilon)\times B(\bar u,
    \epsilon))\cap\Phi$, one has
    $$
    I(x, u)-I(\bar x, \bar u)\notin -{\rm int}\, \mathbb{R}^m_+.
    $$
    \item [(ii)] $\bar z$ is a  locally Pareto solution of the (MCP) if there exists $\epsilon>0$ such
    that for all $(x, u)\in (B(\bar x, \epsilon)\times B(\bar u,
    \epsilon))\cap\Phi$, one has
    $$
    I(x, u)-I(\bar x, \bar u)\notin - \mathbb{R}^m_+\setminus\{0\}.
    $$
\end{enumerate}

}
\end{definition}

Let us denote by $\mathcal{C}_0(\bar z)$ the set of vectors $z=(x,
u)\in C([0, 1], \mathbb{R}^n)\times L^\infty([0, 1], \mathbb{R}^l)$
such that the following conditions hold:
\begin{enumerate}
    \item [($c_1$)] $\int_0^1 \left(L_x[t]x(t) +L_u(t) u(t)\right)dt\in
    - \mathbb{R}^m_+$;
    \item [($c_2$)]
    $x(\cdot)=\int_0^{(\cdot)}\big(\varphi_x[s]x(s)+\varphi_u[s]u(s)\big)ds$;
     \item [($c_3$)] $g_x[\cdot]x+ g_u[\cdot]u\in {\rm
     cone}(Q-g(\cdot,\bar x, \bar u))$.
\end{enumerate} Let  $\mathcal{C}(\bar
z)$ be the closure of $\mathcal{C}_0(\bar z)$ in $C([0, 1],
\mathbb{R}^n)\times L^\infty([0, 1], \mathbb{R}^l)$. We call
$\mathcal{C}(\bar z)$ the {\em critical cone} of the (MCP) at $\bar z$. Each
vector $z\in \mathcal{C}(\bar z)$ is called a {\em critical direction} to
the (MCP) at $\bar z$. It is easily seen  that $\mathcal{C}(\bar z)$ is
a closed convex cone containing $0$.

The following theorem gives necessary optimality conditions for the
(MCP).

\begin{theorem}\label{SOC-necessary-condition} Suppose that assumptions $(H1)$ and $(H2)$ are valid and
$\bar z$ is a locally weak Pareto solution of the (MCP). Then, for each $z\in\mathcal{C}(\bar z)$, there exist a vector
$\lambda\in \mathbb{R}^m_+$ with $|\lambda|=1$, an absolutely
continuous function $\bar p\colon [0, 1]\to \mathbb{R}^n$ and a
function $\theta\in L^1([0, 1], \mathbb{R})$ such that the
following conditions are fulfilled:
\begin{enumerate}
    \item [(i)] $\theta\in N(Q, g[\cdot]);$

    \item [(ii)] (the adjoint equation)
    \begin{equation}\notag
    \begin{cases}
    \dot{\bar p}(t)= - \lambda^T L_x[t]-\varphi_x[t]\bar p(t)-
    \theta(t)g_x[t] \ \ \text{a.e.} \ \ t\in [0, 1],\\
    \bar p(1)=0;
    \end{cases}
    \end{equation}

    \item [(iii)] (the stationary condition in $u$)
    $$
    \lambda^T L_u[t] +\bar p^T \varphi_u[t] +\theta(t)
    g_u[t]=0 \ \ \text{a.e.} \ \ t\in [0, 1];
    $$

    \item [(iv)] (the non-negative condition)
    \begin{align*}
    \int_0^1\left(\sum_{j=1}^m\lambda_j\nabla^2 L_j[t]z(t),z(t)\right)dt&+ \int_0^1
    \left(\bar p(t)^T\nabla^2\varphi[t]z(t), z(t)\right)dt
    \\
    &+\int_0^1\left(\theta(t)\nabla^2 g[t]z(t), z(t)\right)dt\geq 0.
    \end{align*}
\end{enumerate}
\end{theorem}

A point $\bar z\in\Phi$ satisfying conditions (i)--(iii) of Theorem \ref{SOC-necessary-condition} w.r.t $(\lambda, \bar p, \theta)$  is called a KKT point of the (MCP).

In multi-objective optimization problems, the critical cone for
second-order sufficient conditions is often required  bigger than
the one for second-order necessary conditions. Therefore, we need to
enlarge $\mathcal{C}(\bar z)$ to deal with second-order sufficient
conditions. We denote by $\mathcal{C}'(\bar z)$ the set of
vectors $(x, u)\in C([0, 1], \mathbb{R}^n)\times L^2([0, 1],
\mathbb{R}^l)$ which satisfy the following conditions:
\begin{enumerate}
    \item [($c_1^\prime$)] $\int_0^1 \left(L_x[t]x(t) +L_u[t] u(t)\right)dt\in
    -\mathbb{R}^m_+$;

    \item [($c_2^\prime$)] $x(\cdot)=\int_0^{(\cdot)}\big(\varphi_x[s]x(s)+\varphi_u[s]u(s)\big)ds$;

    \item [($c_3^\prime$)] $g_x[t]x(t)+ g_u[t]u(t)\in T((-\infty, 0]; g[t])$
    for a.e. $t\in[0, 1]$.
\end{enumerate}

Obviously, $\mathcal{C}'(\bar z)$ is a closed convex cone and
$\mathcal{C}(\bar z)\subset \mathcal{C}'(\bar z)$.

We now introduce the concept of locally strong Pareto solution for the  multi-objective optimal control problem  (MCP).
\begin{definition}\label{Def2}{ Let $\bar z=(\bar x, \bar u)\in \Phi$ be a feasible point of the (MCP). We say that $\bar z$ is a  locally strong Pareto solution of the (MCP) if there exist a number
$\epsilon>0$ and a vector $c\in \mathrm{int}\,\mathbb{R}^m_+$ such that
for all $(x, u)\in (B_X(\bar x, \epsilon)\times B_U(\bar u,
\epsilon))\cap\Phi$, one has
$$
I(x, u)-I(\bar x, \bar u)-c\|u-\bar u\|^2_2\notin
-\mathbb{R}^m_+\setminus\{0\}.
$$
}
\end{definition} Clearly, every locally strong Pareto  solution of the (MCP) is
 also a locally  Pareto  solution of this problem. Note that in
Definition \ref{Def2} we use two norms to define locally strong
Pareto solutions. Here $B_U(\bar u, \epsilon))$  is a ball  in
$L^\infty([0, 1], \mathbb{R}^l)$ while $\|u-\bar u\|_2$ is the norm
in $L^2([0, 1], \mathbb{R}^l)$.

The following theorem provides sufficient conditions for locally
strong Pareto solutions.

\begin{theorem} \label{SOC-sufficient-condition} Suppose that assumptions $(H1)$ and $(H2)$ are valid and
    $\bar z\in \Phi$. Suppose that there exist a vector $\lambda\in \mathbb{R}^m_+$
    with $|\lambda|=1$, an absolutely continuous function $\bar p \colon [0,
    1]\to \mathbb{R}^n$ and a function $\theta\in L^2([0, 1],
    \mathbb{R})$ satisfy conditions $(i), (ii)$ and $(iii)$ of Theorem
    \ref{SOC-necessary-condition} and the strict second-order condition
    \begin{align}
    \int_0^1\left(\lambda^T\nabla^2 L_j[t]z(t),z(t)\right)dt&+ \int_0^1 \left(\bar
    p(t)^T\nabla^2\varphi[t]z(t), z(t)\right)dt\notag \\
    &+\int_0^1\left(\theta(t)\nabla^2
    g[t]z(t), z(t)\right)dt>0 \label{StrSOC}
    \end{align} for all $z=(x, u)\in\mathcal{C}'(\bar z)\setminus\{0\}$.
    Furthermore,  there is a number $\gamma_0>0$ such that
    \begin{equation}\label{StrSOC1}
    \lambda^TL_{uu}[t](\xi, \xi)\geq \gamma_0|\xi|^2\ \ \forall
    \xi\in\mathbb{R}^l.
    \end{equation}
    Then, $(\bar x, \bar u)$ is a locally strong Pareto solution of the
    (MCP).
\end{theorem}

\section{Abstract multi-objective optimization}
\label{abstract-MP}
Assume that $Z$ and $E$ are Banach spaces with the dual spaces $Z^*$
and $E^*$, respectively.  We consider the following multi-objective
optimization problem:
\begin{align}
&\textrm{Min}_{\,\mathbb{R}^m_+}\ \  f(z)=(f_1(z), \ldots, f_m(z))
\tag{MP1}\label{problem}
\\
&\textrm{s.t.}\quad \quad G(z)\in Q,\notag
\end{align}
where  $f\colon Z\to\mathbb{R}^m$ and $G\colon Z\to E$ are of class
$C^2$,  and $Q$ is a nonempty closed convex subset in $E$. We
denote by $\Sigma$ the feasible set of the \eqref{problem}, that is,
$$
\Sigma=\{z\in Z\;|\; G(z)\in Q\}.
$$
To derive optimality conditions for the \eqref{problem} we need some concepts
of variational analysis. Let $X$ be a Banach space with the dual
$X^*$,  $B_X$ and $B_X(x, r)$ stand for the closed unit ball and the
closed ball with center $x$ and radius $r$, respectively.  Given a subset $A$ of $X$, we denote the interior and the closure of $A$ respectively
by $\mathrm{int}\,A$ and $\overline{A}$.

Let $\Omega$ be a nonempty and closed subset in $X$ and $\bar x\in
\Omega$. The sets
 \begin{align*}
 T^\flat(\Omega; \bar x)&:=\left\{h\in X\;|\; \forall t_k\to 0^+, \exists h_k\to h, \bar x+t_kh_k\in \Omega \ \ \forall k\in\mathbb{N}\right\},
 \\
 T(\Omega; \bar x)&:=\left\{h\in X\;|\; \exists t_k\to 0^+, \exists h_k\to h, \bar x+t_kh_k\in \Omega \ \ \forall k\in\mathbb{N}\right\},
 \end{align*}
 are called {\em the adjacent tangent cone} and {\em the contingent cone} to $\Omega$ at $\bar x$, respectively. It is well-known that when $\Omega$ is convex, then
 $$T^\flat(\Omega; \bar x)=T (\Omega; \bar x)=\overline{\Omega(\bar x)},$$
 where $\Omega(\bar z)$ is defined as follows
 $$\Omega(\bar x):=\mathrm{cone}\,(\Omega-\bar x)=\{\lambda (h-\bar x)\;|\; h\in\Omega, \lambda>0\}.$$
Let $\bar x\in\Omega$ and $h\in X$. The sets
\begin{align*}
T^{2\flat}(\Omega; \bar x, h)&:=\left\{w\in X\;|\; \forall t_k\to 0^+, \exists w_k\to w, \bar x+t_kh+\frac{1}{2}t^2_kw_k\in \Omega\ \ \forall k\in\mathbb{N}\right\},
\\
T^2(\Omega; \bar x, h)&:=\left\{w\in X\;|\; \exists t_k\to 0^+, \exists w_k\to w, \bar x+t_kh+\frac{1}{2}t^2_kw_k\in \Omega\ \ \forall k\in\mathbb{N}\right\},
\end{align*}
are called {\em the second-order adjacent tangent set} and {\em the second-order contingent tangent set} to $\Omega$ at $\bar x$ in the direction $h$, respectively. Clearly, $T^{2\flat}(\Omega; \bar x, h)$ and $T^2(\Omega; \bar x, h)$ are closed sets and
\begin{equation*}
T^{2\flat}(\Omega; \bar x, h)\subset T^2 (\Omega; \bar x, h), T^{2\flat}(\Omega; \bar x, 0)=T^\flat(\Omega; \bar x), T^2 (\Omega; \bar x, 0)=T(\Omega; x).
\end{equation*}
It is noted that if $\Omega$ is convex, then so is
$T^{2\flat}(\Omega; \bar x, h)$. However, $T^2 (\Omega; \bar x, h)$
may not be convex when $\Omega$ is convex (see, for example,
\cite{Bonnans2000}).  In the case that $\Omega$ is a convex set, the
normal cone to $\Omega$ at $\bar x$ is defined by
\begin{equation*}
N(\Omega; \bar x):=\{x^*\in X^*\;|\; \langle x^*, x-\bar x\rangle\leq 0 \ \ \forall x\in\Omega\},
\end{equation*}
or, equivalently,
\begin{equation*}
N(\Omega; \bar x)=\{x^*\in X^*\;|\; \langle x^*, h\rangle\leq 0\ \ \forall h\in T(\Omega; \bar x)\}.
\end{equation*}

\begin{definition}{
We say that $\bar z\in \Sigma$ is a locally  weak  Pareto solution
of the \eqref{problem} if there exists $\epsilon>0$  such that for all
$z\in B_Z(\bar z, \epsilon)\cap \Sigma$, one has $f(z)-f(\bar
z)\notin - \mathrm{int}\,\mathbb{R}^m_+$. }
\end{definition}

We say that the {\em Robinson constraint qualification} holds at
$\bar z\in \Sigma$ if the following condition is verified
\begin{equation*}\label{RC1}
0\in \mathrm{int}\, \left[\nabla G(\bar z)(B_Z)-(Q-G(\bar z))\cap B_E\right].
\end{equation*} According to \cite[Theorem 2.1]{Zowe}, the Robinson
constraint qualification is equivalent to the following condition:
\begin{equation*}
E=\nabla G(\bar z)Z-{\rm cone}(Q-G(\bar z)).
\end{equation*} When the Robinson constraint
qualification holds at $\bar z$, we say $\bar z$ is a {\em regular
point} of the \eqref{problem}.  Hereafter we always assume that $\bar z$
is a feasible regular point of the \eqref{problem}.

Let us define the following critical cones
\begin{align*}
\mathcal{C}_1(\bar z)&=\{d\in Z\;|\; \nabla f(\bar z) d\in
-\mathbb{R}^m_+, \nabla G(\bar z)d\in T(Q; G(\bar z))\},
\\
\mathcal{C}_{01}(\bar z)&=\{d\in Z\;|\; \nabla f(\bar z) d\in
-\mathbb{R}^m_+, \nabla G(\bar z)d\in \text{cone}\,(Q-G(\bar
z))\},\\
\mathcal{C}_{1*}(\bar z)&=\overline{\mathcal{C}_{01}(\bar z)}.
\end{align*}
For each $d\in \mathcal{C}(\bar z)$, put
$$I(\bar z, d)=\{i\in I\,|\, \nabla f_i(\bar z)d=0\},$$
where  $I:=\{1, \ldots, m\}$.
\begin{lemma}\label{firt-order-lemma} If $\bar z$ is a locally weak  Pareto solution of the \eqref{problem}, then
    \begin{equation*}
    \mathcal{C}_1(\bar z)=\{d\in Z\;|\; \nabla f(\bar z) d\in -\mathbb{R}^m_+\setminus({\rm -int}\,\mathbb{R}^m_+), \nabla G(\bar z)d\in T(Q; G(\bar z))\}.
    \end{equation*}
Consequently, the following set
$$\{d\in Z\;|\; \nabla f(\bar z) d\in -\mathrm{int}\,\mathbb{R}^m_+, \nabla G(\bar z)d\in T(Q; G(\bar z))\}$$
is empty.
\end{lemma}
\begin{proof}
Thanks to \cite[Theorem 3.1]{cominetti90}, we have
$$
\nabla G(\bar z)^{-1}(T(Q; G(\bar z)))=T(G^{-1}(Q); \bar
z)=T(\Sigma; \bar z).
$$
Take $d\in \mathcal{C}_1(\bar z)$, then $d\in T(\Sigma; \bar z)$.
Hence there exists a sequence $\{(t_k, d_k)\}$ converging to $(0^+,
d)$ such that $\bar z+t_kd_k\in \Sigma$ for all $k\in\mathbb{N}$.
Since $\bar z$ is a locally weak  Pareto solution of the
\eqref{problem}, we may assume that
$$f(\bar z+t_kd_k)-f(\bar z)\in \mathbb{R}^m\setminus (-\text{int}\,\mathbb{R}^m_+)$$
for all $k\in \mathbb{N}$. From Mean Value Theorem for differentiable functions, we have
$$t_k\nabla f(\bar z)d_k+o(t_k)\in \mathbb{R}^m\setminus (-\text{int}\,\mathbb{R}^m_+),$$
or, equivalently,
\begin{equation*}\label{equa:4}
\nabla f(\bar z)d_k+\frac{o(t_k)}{t_k}\in \mathbb{R}^m\setminus (-\text{int}\,\mathbb{R}^m_+)
\end{equation*}
for all $k\in \mathbb{N}$. Letting $k\to\infty$, by the closedness of $\mathbb{R}^m\setminus (-\text{int}\,\mathbb{R}^m_+)$, we get
$$\nabla f(\bar z)d\in \mathbb{R}^m\setminus (-\text{int}\,\mathbb{R}^m_+).$$
The proof is complete.
\end{proof}
\begin{lemma} Let $\bar z$ be a locally weak  Pareto solution of the \eqref{problem} and $\Lambda_1(\bar z)$ be
the set of normalized Karush--Kuhn--Tucker  multipliers of the
\eqref{problem} at $\bar z$, that is,
\begin{equation*}
\Lambda_1(\bar z):=\{(\lambda, e^*)\in \mathbb{R}^m_+\times E^*\;|\;
\langle \lambda, \nabla f(\bar z)\rangle+ \nabla G(\bar z)^* e^*=0,
|\lambda|=1, e^*\in N(Q; G(\bar z))\}.
\end{equation*} Then $\Lambda_1(\bar z)$ is a  nonempty bounded and compact set in $\mathbb{R}^m\times E^*$ with
respect to topology $\tau_{\mathbb{R}^m}\times\tau(E^*, E)$, where
$\tau(E^*, E)$ is the weakly star topology in $E^*$.
\end{lemma}
\begin{proof} We first claim that $\Lambda_1(\bar z)$ is nonempty. Indeed, put
\begin{equation*}
\Psi=\{(\nabla f_1(\bar z)d+r_1, \ldots, \nabla f_m(\bar z)d+r_m,
\nabla G(\bar z)d-v)\,|\, d\in Z, v\in T (Q; G(\bar z)), r_i\geq 0,
i\in I\}.
\end{equation*}
Then, $\Psi$ is a convex subset in $\mathbb{R}^m\times E$. By the Robinson constraint qualification, there exists $\rho>0$ such that
\begin{equation}\label{equa-Robinson}
B_E(0, \rho)\subset \nabla G(\bar z)(B_Z)-(Q-G(\bar z))\cap B_E.
\end{equation}
This implies that
\begin{equation*}
B_E(0, \rho)\subset \nabla G(\bar z)(B_Z)-T(Q;G(\bar z))\cap B_E.
\end{equation*}
For each $i\in I$, put
$$\alpha_i=\|\nabla f_i(\bar z)\|=\sup_{d\in B_Z}|\nabla f_i(\bar z)d|.$$
It is easily seen that
\begin{equation*}
(\alpha_1, +\infty)\times\ldots\times (\alpha_m, +\infty)\times B(0, \rho)\subset \Psi.
\end{equation*}
Thus, $\Psi$ has  a nonempty interior. We show that $(0, 0)\notin \text{int}\,\Psi$. If otherwise, there exist $\epsilon_1>0$,  $\ldots$, $\epsilon_m>0$ such that
$$(-\epsilon_1, \epsilon_1)\times \ldots\times (-\epsilon_m, \epsilon_m)\times \{0\}\subset \Psi.$$
This implies that there exist $d\in Z$, $v\in T(Q; G(\bar z))$, and $r_i\geq 0$, $i\in I$ such that
\begin{equation*}\label{equa:1}
\begin{cases}
&\nabla f_i(\bar z)d+r_i<0, \ \ i\in I,
\\
&\nabla G(\bar z)d-v=0.
\end{cases}
\end{equation*}
Consequently, the system
\begin{equation*}
\begin{cases}
&\nabla f_i(\bar z)d<0, \ \ i\in I,
\\
&\nabla G(\bar z)d\in T(Q; G(\bar z)),
\end{cases}
\end{equation*}
has at least one solution $d\in Z$, which contradicts the conclusion
of Lemma \ref{firt-order-lemma}. We now can separate $(0, 0)$ from
$\Psi$ by a hyperplane, i.e., there exists a functional $(\lambda,
e^*)\in (\mathbb{R}^m\times E^*)\setminus\{(0, 0)\}$ such that
\begin{equation}\label{equa:2}
\sum_{i=1}^{m}\lambda_i(\nabla f_i(\bar z)d+r_i) + \langle e^*, \nabla G(\bar z)d-v\rangle\geq 0
\end{equation}
for all $d\in Z$, $v\in T(Q; G(\bar z))$, $r_i\geq 0$, $i\in I$. From \eqref{equa:2} it follows that $\lambda_i\geq 0$ for all $i\in I$, that is,  $\lambda\in \mathbb{R}^m_+$. Putting $v=0$ and $r_i=0$, $i\in I$ into \eqref{equa:2}, we get
\begin{equation*}
(\langle\lambda, \nabla f(\bar z)\rangle+\nabla G(\bar z)^*e^*)d\geq 0 \ \ \forall d\in Z.
\end{equation*}
Thus, $\langle\lambda, \nabla f(\bar z)\rangle+\nabla G(\bar z)^*e^*=0$. Putting this equation and $r_i=0$, $i\in I$, into \eqref{equa:2}, one has
\begin{equation*}
e^*(v)\leq 0 \ \ \forall v\in T(Q; G(\bar z)).
\end{equation*}
Hence, $e^*\in N(Q; G(\bar z))$. We now show that $\lambda\neq 0$. Indeed, if otherwise, then we have
\begin{equation}\label{equa:3}
\langle e^*, \nabla G(\bar z)d-v\rangle\geq 0
\end{equation}
for all $d\in Z$, $v\in T(Q; G(\bar z))$. Again by the Robinson constraint qualification, one has
$$E=\nabla G(\bar z)(Z)-Q(G(\bar z)).$$
This and \eqref{equa:3} imply that $e^*=0$, a contradiction.
Put $(\bar \lambda, \bar e^*)=\left(\frac{\lambda}{|\lambda|},
\frac{e^*}{|\lambda|}\right)$. Then we have $(\bar \lambda, \bar
e^*)\in \Lambda_1(\bar z)$, as required.

We now claim that $\Lambda_1(\bar z)$ is bounded. Indeed, fix $(\lambda_0,
e^*_0)\in \Lambda_1(\bar z)$. Then, for any $(\lambda, e^*)$ belonging to $\Lambda_1(\bar z)$, we have
\begin{equation*}
\begin{cases}
&\langle \lambda_0, \nabla f(\bar z)\rangle+\nabla G(\bar z)^*e^*_0=0, \ \ e^*_0\in N(Q; G(\bar z)),
\\
&\langle \lambda, \nabla f(\bar z)\rangle+\nabla G(\bar z)^*e^*=0, \ \ e^*\in N(Q; G(\bar z)).
\end{cases}
\end{equation*}
Since \eqref{equa-Robinson}, for each $y\in B_E(0, \rho)$, there exist $z\in B_Z$ and $w\in (Q-G(\bar z))\cap B_E$ such that $y=\nabla G(\bar z) z-w$.
It follows that
\begin{align*}
\langle e^*_0-e^*, y\rangle&=\langle e^*_0-e^*, \nabla G(\bar z) z-w\rangle
\\
&=\langle \nabla G(\bar z)^*(e^*_0-e^*), z\rangle-\langle e^*_0-e^*, w\rangle
\\
&=\langle (\lambda-\lambda_0) \nabla f(\bar z), z\rangle-\langle e^*_0, w\rangle+\langle e^*, w\rangle
\\
&\leq 2\|\nabla f(\bar z)\|+\|e^*_0\|.
\end{align*}
Thus,
\begin{align*}
\langle -e^*, y\rangle&\leq 2\|\nabla f(\bar z)\|+\|e^*_0\|+\|e^*_0\|\|y\|
\\
&\leq 2\|\nabla f(\bar z)\|+\|e^*_0\|+\|e^*_0\|\rho.
\end{align*}
Replacing $y$ by $\rho y$ with $\|y\|\leq 1$, we get
$$\|e^*\|\rho\leq 2\|\nabla f(\bar z)\|+\|e^*_0\| (1+\rho).
$$
Consequently,
$$\|e^*\| \leq \frac{2}{\rho}\|\nabla f(\bar z)\|+\frac{1+\rho}{\rho}\|e^*_0\|.
$$
Thus, $\Lambda_1(\bar z)$ is bounded. It is easy to
check that the set $\Lambda_1(\bar z)$ is closed with respect to topology
$\tau_{\mathbb{R}^m}\times\tau(E^*, E)$. Thanks to  \cite[Theorem 3.16]{Brezis1},
$\Lambda_1(\bar z)$  is  compact.
\end{proof}

To derive second-order necessary conditions for the
\eqref{problem}, we need  the following result.

\begin{lemma}\label{abstract-necessary}  Let $\bar z\in \Sigma$ and $d\in \mathcal{C}_1(\bar z)$. If $\bar{z}$ is a locally weak  Pareto solution of the \eqref{problem}, then the following system
 \begin{align}
&\nabla f_i(\bar z)z+\nabla^2f_i(\bar z)d^2< 0, \ \ \forall i\in I(\bar z, d), \label{second-order-abstract-1}
\\
&\nabla G(\bar z)z+\nabla^2G(\bar z)d^2\in T^{2\flat}(Q; G(\bar z), \nabla G(\bar z)d), \label{second-order-abstract-2}
\end{align}
has no solution $z\in Z$.
\end{lemma}
\begin{proof} Arguing by contradiction, assume that the system \eqref{second-order-abstract-1}--\eqref{second-order-abstract-2} admits a solution, say $z$. From \eqref{second-order-abstract-2} it follows that
$$
z\in \nabla G(\bar z)^{-1}\left[T^{2\flat}(Q; G(\bar z), \nabla G(\bar z)d)-\nabla^2 G(\bar z)d^2\right].
$$
By the Robinson constraint qualification and \cite[Theorem 3.1]{cominetti90},  we have
$$
T^{2\flat}(\Sigma; \bar z, d)=\nabla G(\bar
z)^{-1}\left[T^{2\flat}(Q; G(\bar z), \nabla G(\bar z)d)-\nabla^2
G(\bar z)d^2\right].
$$
Thus, $z\in T^{2\flat}(\Sigma; \bar z, d)$. Let $\{t_k\}$ be an
arbitrary sequence converging to $0^+$. Then, there exists a
sequence $\{z_k\}$ tending to $z$ such that
$$w_k:=\bar z+t_kd+\frac{1}{2}t_k^2z_k\in \Sigma \ \  \forall k\in\mathbb{N}.$$
For each $i\in I(\bar z, d)$ and $k\in \mathbb{N}$, we have
\begin{align*}
f_i(w_k)-f_i(\bar z)&=[f_i(w_k)-f_i(\bar z+t_kd)]+[f_i(\bar z+t_kd)-f_i(\bar z)-t_k\langle \nabla f_i(\bar z), d\rangle]
\\
&= \frac{1}{2}t^2_k\langle\nabla f_i(\bar z+ t_k d),
z_k\rangle+\frac{1}{2}t^2_k \nabla^2 f_i(\bar z) d^2+o(t^2_k).
\end{align*}
Therefore,
\begin{equation*}
\lim\limits_{k\to\infty} \frac{f_i(w_k)-f_i(\bar
z)}{\frac{1}{2}t^2_k}=\langle\nabla f_i(\bar z), z\rangle+\nabla^2 f_i(\bar
z)d^2.
\end{equation*}
This and \eqref{second-order-abstract-1} imply that
\begin{equation*}
f_i(w_k)<f_i(\bar z)
\end{equation*}
for all $i\in I(\bar z, d)$ and $k$ large enough. For each $i\in I\setminus I(\bar z, d)$, we have $\langle\nabla f_i(\bar z), d\rangle<0$. From this and the fact that
\begin{equation*}
\lim\limits_{k\to\infty}\frac{f_i(w_k)-f_i(\bar z)}{t_k}=\langle\nabla f_i(\bar z), d\rangle,
\end{equation*}
it follows that
\begin{equation*}
f_i(w_k)<f_i(\bar z)
\end{equation*}
for all $k$ large enough. Thus there exists $k$ large enough such that
\begin{equation*}
f_i(w_k)<f_i(\bar z) \ \ \forall i\in I,
\end{equation*}
which contradicts the fact that $\bar{z}$ is a locally weak  Pareto solution of the \eqref{problem}.
\end{proof}

Problem \eqref{problem} is associated with the Lagrangian $\mathcal{L}_1(z,
\lambda, e^*)=\lambda^T f(z)+ e^* G(z)$. The following theorem gives some
second-order necessary optimality conditions for the \eqref{problem}.

\begin{theorem}\label{second-order-necessary-cond}
 Suppose that $\bar{z}$ is a locally weak  Pareto solution of the \eqref{problem}. Then, for each $d\in \mathcal{C}_{1*}(\bar z)$,
 there exists $(\lambda, e^*)\in\Lambda_1(\bar z)$
 such that the  following non-negative second-order condition is valid:
\begin{equation*}
\nabla^2_{zz}\mathcal{L}_1(\bar z, \lambda, e^*)(d, d)\geq 0.
\end{equation*}
\end{theorem}
\begin{proof} We first prove the theorem for  $d\in \mathcal{C}_{01}(\bar z)$.
We claim that  $0\in T^{2\flat}(Q; G(\bar z), \nabla G(\bar z)d)$. Indeed, since $\nabla G(\bar z)d\in \mathrm{cone}\, (Q-G(\bar z))$, there exists  $\mu>0$ such that $\mu \nabla G(\bar z)d\in Q-G(\bar z)$. From $0\in Q-G(\bar z)$ and the convexity of $Q-G(\bar z)$, for any $0<\alpha<\mu$, we get
\begin{equation*}
\frac{\alpha}{\mu}\mu \nabla G(\bar z)d +\left(1-\frac{\alpha}{\mu}\right).0\in Q-G(\bar z).
\end{equation*}
This implies that $G(\bar z)+\alpha\nabla G(\bar z) d\in Q$ for all
$\alpha\in (0, \mu)$. Hence, $0\in T^{2\flat}(Q; G(\bar z), \nabla
G(\bar z)d)$ as required.   We consider the following set
\begin{align*}
\Pi:=&\big\{(\nabla^2f_1(\bar z)d^2+\nabla f_1(\bar z)z+r_1,...,
\nabla^2f_m(\bar z)d^2+\nabla f_m(\bar z)z+r_m, \nabla^2G(\bar
z)d^2+\nabla G(\bar z)z-v)|
\\
 &z\in Z, v\in T^{2\flat}(Q; G(\bar z), \nabla G(\bar z)d), r_i\geq 0, i\in I\big\}.
\end{align*}
Thanks to Lemma \ref{firt-order-lemma}, the set $I(\bar z, d)$ is nonempty.
Thus, by choosing $\lambda_i=0$ and removing the component $i$-th of
$\Pi$ for $i\in I\setminus I(\bar z, d)$, we may assume that
$I(\bar z, d)=I$. We claim that $\Pi$ is a convex set with a nonempty interior and $(0, 0)\notin\mathrm{int}\, \Pi$. The convexity of
$\Pi$ follows directly from the convexity of $T^{2\flat}(Q; G(\bar z),
\nabla G(\bar z)d)$ and $\mathbb{R}^m_+$. By the Robinson constraint qualification, there exists $\rho>0$ such that
\begin{equation*}
B_E(0, \rho)\subset \nabla G(\bar z)[B_Z]-T^{\flat}(Q; G(\bar z))\cap B_E.
\end{equation*}
Since $0 \in T^{2\flat}(Q; G(\bar z), \nabla G(\bar z)d)$ and \cite[Proposition 3.1]{cominetti90}, we have
\begin{equation*}
T^{2\flat}(Q; G(\bar z), \nabla G(\bar z))=T^{\flat}(T^\flat(Q; G(\bar z)), \nabla G(\bar z)).
\end{equation*}
Putting $V=\nabla G(\bar z)+B_E(0, \rho)$ and $\rho_1=1+\|\nabla
G(\bar z)d\|$, we have
\begin{align*}
V&\subset \nabla G(\bar z)[B_Z]-\left[T^{\flat}(Q;G(\bar z))\cap B_E-\nabla G(\bar z)\right]
\\
&\subset\nabla G(\bar z)[B_Z]-\left[T^{\flat}(Q; G(\bar z))-\nabla G(\bar z)\right]\cap B_E(0, \rho_1)
\\
&\subset\nabla G(\bar z)[B_Z]-T^{\flat}(T^\flat(Q; G(\bar z)), \nabla G(\bar z))\cap B_E(0, \rho_1)
\\
&=\nabla G(\bar z)[B_Z]-T^{2\flat}(Q; G(\bar z), \nabla G(\bar z))\cap B_E(0, \rho_1).
\end{align*}
By the Robinson constraint qualification,
\begin{equation}\label{sujective_property}
E=\nabla G(\bar z)[Z]-Q(G(\bar z))=\nabla G(\bar z)[Z]-T^{2\flat}(Q; G(\bar z), \nabla G(\bar z)d).
\end{equation}
For each $i\in I$, we put
$$\alpha_i=\nabla^2 f_i(\bar z)d^2+\sup \left\{\nabla f_i(\bar z)z\,|\, z\in  B_Z(0,\rho_1)\right\}<+\infty.$$
We then have
$$
(\alpha_1, +\infty)\times (\alpha_2, +\infty)\times \ldots \times(\alpha_m, +\infty)\times \hat V\subset \Pi,
$$
where $\hat V:=\nabla^2G(\bar z)d^2+V$. This implies that the interior of
$\Pi$ is nonempty. We now show that $(0, 0)\notin\mathrm{int}\, \Pi$. If
otherwise, there exist $\epsilon_1>0$, $\epsilon_2>0$, $\ldots$,
$\epsilon_m>0$ such that
$$
(-\epsilon_1, \epsilon_1)\times(-\epsilon_2, \epsilon_2)\times\ldots\times (-\epsilon_m, \epsilon_m)\times \{0\}\subset \Pi.
$$
This implies that there exist $z\in Z$, $v\in T^{2\flat}(Q; G(\bar z), \nabla G(\bar z)d)$, $r_i\geq 0$, $i\in I$, satisfying
$$
\begin{cases}
\nabla^2f_i(\bar z)d^2+\nabla f_i(\bar z)z+r_i<0, \ \ i\in I,
\\
\nabla^2G(\bar z)d^2+\nabla G(\bar z)z-v=0.
\end{cases}
$$
Consequently, $z$ is a solution of the following system
$$
\begin{cases}
\nabla^2f_i(\bar z)d^2+\nabla f_i(\bar z)z<0, \ \ i\in I,
\\
\nabla^2G(\bar z)d^2+\nabla G(\bar z)z\in T^{2\flat}(Q; G(\bar z), \nabla G(\bar z)d),
\end{cases}
$$
contrary to Lemma \ref{abstract-necessary}. Since  $(0, 0)\notin\mathrm{int}\, \Pi$, we can separate $(0, 0)$ from $\Pi$ by a
hyperplane, i.e., there exists a functional $(\lambda, e^*)\in
(\mathbb{R}^m\times E^*)\setminus\{(0,0)\}$ such that
\begin{equation}\label{separate-property}
\sum_{i=1}^{m}\lambda_i(\nabla^2f_i(\bar z)d^2+\nabla f_i(\bar z)z+r_i)+\langle e^*, \nabla^2G(\bar z)d^2+\nabla G(\bar z)z-v\rangle\geq 0
\end{equation}
for all $r_i\geq 0$, $i\in I$, $z\in Z$, and $v\in T^{2\flat}(Q; G(\bar z), \nabla G(\bar z)d)$. By \eqref{separate-property}, we have $\lambda_i\geq 0$ for all $i\in I$, i.e., $\lambda\in\mathbb{R}^m_+$. We claim that $\lambda$ is a nonzero vector. If otherwise, then we have
$$\langle e^*, \nabla^2 G(\bar z)d^2+\nabla G(\bar z)z-v\rangle\geq 0$$
for all $z\in Z$, and $v\in T^{2\flat}(Q; G(\bar z), \nabla G(\bar z)d)$, or, equivalently,
$$\langle e^*, \nabla^2G(\bar z)d^2+w\rangle\geq 0$$
for all $w\in \nabla G(\bar z)(Z)-T^{2\flat}(Q; G(\bar z), \nabla
G(\bar z)d)$. This and \eqref{sujective_property} imply that
$e^*=0$, contrary to the fact that $(\lambda, e^*)\neq (0, 0)$. We
now rewrite \eqref{separate-property} as follows
\begin{equation}\label{separate-property-2}
\langle\nabla_z\mathcal{L}_1(\bar z, \lambda, e^*) , z\rangle +
\nabla_{zz}^2\mathcal{L}_1(\bar z, \lambda, e^*)(d, d)
+\sum_{i=1}^{m}\lambda_ir_i-\langle e^*, v\rangle\geq 0
\end{equation}
for all $r_i\geq 0$, $i\in I$, $z\in Z$, and $v\in T^{2\flat}(Q;
G(\bar z), \nabla G(\bar z)d)$. It follows that
$\nabla_z\mathcal{L}_1(\bar z, \lambda, e^*) =0$.
 Putting $\nabla_z\mathcal{L}_1(\bar z, \lambda, e^*) =0$ and $r_i=0$, $i\in I$, into \eqref{separate-property-2},
  we get
\begin{equation*}
\nabla^2_{zz}\mathcal{L}_1(\bar z, \lambda, e^*)(d, d) \geq \langle
e^*, v\rangle
\end{equation*}
for all $v\in T^{2\flat}(Q; G(\bar z), \nabla G(\bar z)d)$. Thus,
\begin{equation}\label{SCOC2}
\nabla^2_{zz} \mathcal{L}_1(\bar z, \lambda, e^*)(d, d)\geq
\sigma(e^*, T^{2\flat}(Q; G(\bar z), \nabla G(\bar z)d))\geq 0.
\end{equation}
By dividing both sides of \eqref{SCOC2} by $|\lambda|$, we obtain that
\begin{equation*}\label{SCOC3}
\sup_{(\lambda, e^*)\in\Lambda_1(\bar z)}\nabla_{zz}^2
\mathcal{L}(\bar z, \lambda, e^*)(d, d)\geq 0.
\end{equation*}

We now take any $d\in \mathcal{C}_{1*}(\bar z)$. Then, there exists
a sequence $\{d_k\}\subset \mathcal{C}_{01}(\bar z)$ converging to
$d$. From what has already been proved, we have
 \begin{equation}\label{SCOC4}
\sup_{(\lambda, e^*)\in\Lambda_1(\bar z)}\nabla_{zz}^2
\mathcal{L}(\bar z, \lambda, e^*)(d_k, d_k) \geq 0.
\end{equation} Since the set $\Lambda_1(\bar z)$ is compact in
topology $\tau_{\mathbb{R}^m}\times \tau(E^*, E)$, the function
$$
\psi (d):=\sup_{(\lambda, e^*)\in\Lambda_1(\bar z)}\nabla_{zz}^2
\mathcal{L}_1(\bar z, \lambda, e^*)(d, d)
$$ is continuous. Letting $k\to\infty$ in \eqref{SCOC4}, we get
\begin{equation*}
\sup_{(\lambda, e^*)\in\Lambda_1(\bar z)}\nabla_{zz}^2
\mathcal{L}_1(\bar z, \lambda, e^*)(d, d) \geq 0.
\end{equation*} Again, by the compactness of $\Lambda_1(\bar z)$, there exists $(\lambda, e^*)\in\Lambda_1(\bar z)$ such that
\begin{equation*}
\nabla_{zz}^2 \mathcal{L}_1(\bar z, \lambda, e^*)(d, d) \geq 0.
\end{equation*} The proof is complete.
\end{proof}

From Theorem 3.1 we want to ask whether the conclusion is still true
if $\mathcal{C}_{1*}(\bar z)$ is replaced by $\mathcal{C}_1(\bar
z)$. Clearly, $\mathcal{C}_{1*}(\bar z)\subseteq \mathcal{C}_1(\bar
z)$. In the case of single-objective $(m=1)$ under assumptions that $Q$
is {\em polyhedric} at $G(\bar z)$ and $\nabla G(\bar z)$ is
surjective, \cite[Proposition 3.54]{Bonnans2000} showed that
$\mathcal{C}_{1}(\bar z)=\mathcal{C}_{1*}(\bar z)$.  However, when
$m>1$, the proof of Proposition 3.54 in \cite{Bonnans2000} is
collapsed. The reason is that the condition $\langle \lambda, \nabla f(\bar z)\rangle=0$ with $\lambda\neq 0$ does not imply $\nabla f(\bar z)=0$. We do not know whether the equality
$\mathcal{C}_1(\bar z)= \mathcal{C}_{1*}(\bar z)$ is valid. Therefore, we
leave here the following conjecture.

\noindent $\bullet$ {\bf Conjecture}: {\em Suppose that $\nabla
G(\bar z)\colon Z\to E$ is surjective and $Q$ is polyhedric at
$G(\bar z)$. If $\bar z$ is a locally weak Pareto solution of
the \eqref{problem}, then  $\mathcal{C}_1(\bar z)= \mathcal{C}_{1*}(\bar
z)$.}

\section{Abstract multi-objective optimal control problems}
\label{abstract-MCP}
Let $E_0, E, X$ and $ U$  be Banach spaces and $Q$ be a nonempty
closed convex set in $E$. Define $Z=X\times U$ and assume that
\begin{align*}
&I \colon X\times U\to \mathbb{R}^m,\\
&F \colon X\times U \to E_0,\\
&G \colon X\times U\to E
\end{align*}
are given mappings.  We consider the following
multi-objective optimal control problem of finding a control $u\in
U$ and the corresponding state $x\in X$ which solve
\begin{align}
&\textrm{Min\,}_{\mathbb{R}^m_+}\quad  I(x, u),\notag  \\
&\textrm{s.t.}\quad\quad   F(x, u)=0, \tag{MP2} \label{problem2} \\
&\quad\quad\quad G(x, u)\in Q.\notag
\end{align} We denote by $\Phi$ the feasible of the \eqref{problem2} and put
$$
D=\{(x, u)\in Z\;|\; F(x, u)=0\}.
$$
Fix $z_0=(y_0, u_0)\in
 \Phi$. We denoted by $\Lambda_2(z_0)$ the set of multipliers $(\lambda, v^*,
e^*)\in \mathbb{R}^m_+\times E_0^*\times E^*$ with  $|\lambda|=1$,
which satisfies the following conditions
\begin{equation*}
\nabla_z\mathcal{L}_2(z, \lambda, v^*, e^*)=0,\ e^* \in N(Q, G(z,
w_0)),
\end{equation*} where $\mathcal{L}_2(z, \lambda, v^*, e^*)$ is the
Lagrangian which is given by
$$
\mathcal{L}_2(z, v^*, e^*)= \langle\lambda, I(z)\rangle+\langle v^*,
F(z)\rangle + \langle e^*, G(z)\rangle.
$$ We also denote by $\mathcal{C}_2(z_0)$ the closure of ${\mathcal C}_{02}(z_0)$
in $Z$, where
\begin{equation}\notag
{\mathcal C}_{02}(z_0):=\left\{d\in Z\;|\; \nabla I(z_0)d\in -\mathbb{R}^m_+,\ \nabla
F(z_0)d=0,\  \nabla G(z_0)d\in {\rm cone}(Q- G(z_0)\right\}.
\end{equation}
The set $\mathcal{C}_2(z_0)$ is called the critical cone of the \eqref{problem2} at $z_0$.

\medskip

Let us introduce the following assumptions:
\begin{enumerate}
    \item [$(A1)$] There exist positive numbers $r_1, r_1'$ such that
    the mapping $I(\cdot, \cdot)$, $F(\cdot, \cdot)$ and $G(\cdot,
    \cdot)$ are twice continuously Fr\'{e}chet differentiable on
    $B_X(x_0, r_1)\times B_U(u_0, r_1')$;
    \item [$(A2)$] The mapping $F_x(z_0)$ is  bijective;
    \item [$(A3)$] $\nabla G(z_0)(T(D; z_0))=E$.
\end{enumerate}

\medskip

From assumptions $(A1)$ and $(A3)$,  we have that $F(\cdot, \cdot)$ is
continuously differentiable on $B_X(x_0, r_1)\times B_U(u_0, r_1')$
and $F_x(z_0)$ is bijective. By the implicit function theorem (see
\cite[Theorem 4.E]{Zeidler1}), there exist balls $B_X(x_0, r_2)$,
$B_U(u_0, r_2')$ with $r_2<r_1, r_2'<r_1'$ such that for each $u\in
B_U(u_0, r_2')$, the equation
$$
F(x, u)=0
$$ has a unique solution $x=\zeta(u)\in B_X(x_0, r_2)$. Moreover,
the mapping
$$\zeta \colon B_U(u_0, r_2')\to B_X(x_0,
r_2)$$ is  of class $C^2$ and $\zeta(u_0)=x_0$. Thus,
\begin{equation}\label{EqF}
F(\zeta(u), u)=0\ \ \forall u\in B_U(u_0, r_2').
\end{equation}
We now define the following mappings:
\begin{align}
& J \colon U\to\mathbb{R}^m, \quad J(u):=J(\zeta(u), u), \notag \\
& H \colon U\to E,\, \quad H(u):= G(\zeta(u), u).\label{H}
\end{align}  Then we can show that  $(x_0, u_0)$ is a locally weak  Pareto solution of the \eqref{problem2} if
and only if $u_0$ is a locally weak  Pareto solution of the following
problem:
\begin{align}
&\textrm{Min\,}_{\mathbb{R}^m_+} J(u)\, \tag{MP3}\label{problem3}\\
& \textrm{s.t.}\ H(u)\in Q.\notag
\end{align} Problem \eqref{problem3} is associated with the Lagrangian
$$
\mathcal{L}_3(u,\lambda, e^*)=\lambda J(u) +e^*H(u).
$$ Given a feasible point $u_0$ of the \eqref{problem3}, we define
\begin{equation*}
\mathcal{C}_{03}(u_0)=\{u\in U \;|\; \nabla J(u_0)u\in -\mathbb{R}^m_+, \nabla
H(u_0)u\in {\rm cone}(Q- H(u_0))\}
\end{equation*}
and $\mathcal{C}_3(u_0)=\overline{\mathcal{C}_{03}(u_0)}$ the interior
critical cone and the critical cone at $u_0$, respectively.

The following theorem provides second-order necessary optimality
conditions for the \eqref{problem2}.

\begin{theorem}\label{ThSOC2}
Suppose that $z_0$ is a feasible point of the  \eqref{problem2} and assumptions
$(A1)-(A3)$ are satisfied. If $z_0$ is a locally weak  Pareto solution of the
\eqref{problem2}, then, for each $d \in \mathcal{C}_2(z_0)$, there exists a
nonzero triple $(\lambda, v^*, e^*) \in \Lambda_2(z_0)$ such that
\begin{equation*}
\nabla_z^2\mathcal{L}_2(z_0, e^*, v^*)(d,d)=\langle \lambda
J_{zz}(z_0)d, d\rangle +\langle v^*F_{zz}(z_0)d, d\rangle +\langle
e^*G_{zz}(z_0)d, d\rangle \geq 0.
\end{equation*}
\end{theorem}
\begin{proof} Since $z_0=(x_0, u_0)$ is a locally weak  Pareto solution to the
\eqref{problem2}, $u_0$  is a locally weak  Pareto solution of the \eqref{problem3}.  By assumption $(A2)$,
$\nabla F(z_0)$ is surjective. Indeed, for any $v\in E_0$, there exists $x\in X$ such that $F_x(\bar z)x=v$.  Hence, $(x, 0)\in X\times U$ and $\nabla F(\bar z)(x, 0)=v$. This means that $\nabla F(\bar z)$ is surjective. From this and \cite[Lemma 2.2]{kien2}, it follows that
$$
T(D; z_0)=\left\{(x, u)\in Z \;|\; F_x(z_0)x+F_u(z_0)u=0\right\}=\left\{(\zeta'(u_0)u,
u) \;|\; u\in U\right\}.
$$
Combining this with $(A3)$, we get
\begin{align*}
E&\subseteq\{\nabla_xG(z_0)\zeta'(u_0)u +\nabla_uG(z_0)u \;|\; u\in
U\}\\
&\subseteq \nabla H(u_0)(U) -{\rm cone}(Q-H(u_0)),
\end{align*}
where $H$ is defined by \eqref{H}. Hence the Robinson constraint qualification for the \eqref{problem3} is satisfied
at $u_0$.

Fix any $d=(x, u)\in \mathcal{C}_2(z_0)$. Then there exists a
sequence $\{d_k\}=\{(x_k, u_k)\}\subset \mathcal{C}_{02}(z_0)$ such that $d_k\to
d$.  Since $F_x(z_0)x_k+ F_u(z_0)u_k=0,$ we have
$x_k=\zeta'(u_0)u_k$ and so $u_k\in \mathcal{C}_{03}(u_0)$.
Consequently, $u\in \mathcal{C}_3(u_0)$.  By Theorem \ref{second-order-necessary-cond}, there exists
a multipliers $(\lambda, e^*)\in \Lambda_*(u_0)$ such that the
following conditions hold:
\begin{enumerate}
    \item [(a)] $\lambda\nabla J(u_0) +\nabla e^*H(u_0)=0$, $e^*\in
    N(Q, H(u_0))$,

    \item [(b)] $\langle \lambda \nabla^2 J(u_0)u, u\rangle
    +\langle e^*\nabla^2H(u_0)u,  u\rangle\geq 0$.
\end{enumerate}

Note that from \eqref{EqF}, we have
\begin{equation}\label{EqF2}
 F(\zeta(v), v)=0\ \ \forall v\in B_U(u_0, r_3').
\end{equation}
Taking first-order derivative on both sides, we get
\begin{equation*}
F_x(z_0)\zeta'(u_0) +F_u(z_0)=0
\end{equation*} and so
\begin{equation}\label{Adj1}
\zeta'(u_0)^*=-F_u(z_0)^* (F_x^*(z_0))^{-1}.
\end{equation}
From (a), we have
$$
 \lambda I_x(z_0)\zeta'(u_0)+ \lambda I_u(z_0) + e^*G_x(z_0)\zeta'(u_0) +
e^*G_u(z_0)=0,
$$
or, equivalently,
\begin{equation}\label{Adj2}
 \zeta'(u_0)^*(\lambda I_x(z_0)+e^*G_x(z_0))=-(\lambda I_u(z_0)+ e^* G_u(z_0)).
\end{equation} Let us put
\begin{equation}\label{Mulp3}
\phi =(\lambda I_x(z_0)+e^*G_x(z_0)),\ v^*=-(F_x^*(z_0))^{-1}\phi.
\end{equation}
Then, from \eqref{Adj1} and \eqref{Adj2}, we have
\begin{equation*}
-F_u(z_0)^* (F_x^*(z_0))^{-1}\phi=-(\lambda I_u(z_0)+ e^* G_u(z_0)).
\end{equation*}
Consequently,
\begin{equation*}
\begin{cases}
&F_u(z_0)^* v^* +\lambda I_u(z_0)+ e^* G_u(z_0)=0\\
&F_x(z_0)^* v^*=-(\lambda I_x(z_0)+e^*G_x(z_0)).
\end{cases}
\end{equation*}
This is equivalent to
\begin{equation*}
\nabla I(z_0)^*\lambda^T+ \nabla F(z_0)^* v^* + \nabla G(z_0)^*
e^*=0.
\end{equation*}
Hence we have $(\lambda, v^*, e^*)\in \Lambda_2(z_0)$.

Let us define the following function
$$
\psi(t):=\mathcal{L}_3(u_0+ tu,\lambda, e^*)=\lambda J (u_0 +tu)
+e^*H(u_0 +tu),\ -1<t<1,\ u\in\mathcal{C}_3(u_0).
$$ Then, by (b), we have
$$
\psi''(0)=\nabla^2_u\mathcal{L}_3(u_0)(u, u)=\langle \lambda \nabla^2
 J(u_0)u, u\rangle +\langle e^*\nabla^2H(u_0)u, u\rangle\geq
0.
$$ On the other hand, by simple calculation, we get
\begin{align}
\psi''(0)=&\langle \lambda I_{xx}(z_0) \zeta'(u_0)u, \zeta'(u_0)u\rangle +
\langle \lambda I_{xu}(z_0)\zeta'(u_0)u, u\rangle + \langle \lambda
I_{ux}(z_0)u, \zeta'(u_0)u\rangle \notag\\
&+\langle \lambda I_{uu}(z_0)u, u\rangle+\langle e^* G_{xx}(z_0) \zeta'(u_0)u, \zeta'(u_0)u\rangle +
\langle e^*G_{xu}(z_0)\zeta'(u_0)u, u\rangle\notag\\
&+ \langle e^*G_{ux}(z_0)u, \zeta'(u_0)u\rangle+\langle e^*G_{uu}(z_0)u, u\rangle+ \langle(\lambda I_x(z_0)+ e^* G_x(z_0))\zeta''(u_0)u, u\rangle\notag
\\
=&\langle \lambda I_{xx}(z_0)x, x\rangle + \langle \lambda
I_{xu}(z_0)x, u\rangle +
\langle \lambda I_{ux}(z_0)u, x\rangle +\langle \lambda I_{uu}(z_0)u, u\rangle\notag\\
&+\langle e^*G_{xx}(z_0)x, x\rangle + \langle e^*G_{xu}(z_0)x,
u\rangle + \langle e^*G_{ux}(z_0)u, x\rangle +\langle
e^*G_{uu}(z_0)u, u\rangle\notag\\
&+ \langle (\zeta''(u_0)u)^*(\lambda I_x(z_0)+ e^* G_x(z_0)),
u\rangle.\label{Mulp4}
\end{align}
Taking second-order derivatives on both sides of \eqref{EqF2} at
$u_0$ and acting on $u\in \mathcal{C}''(u_0)$ and $v\in U$, we
obtain
\begin{align*}
\langle F_x(z_0) \zeta''(u_0)u, v\rangle + \langle
F_{xx}(z_0)\zeta'(u_0)u,
\zeta'(u_0)v\rangle &+\langle F_{xu}(z_0)\zeta'(u_0)u, v\rangle +\\
&+ F_{ux}(z_0)(u,\zeta'(z_0)v)+F_{uu}(z_0)(u, v)=0.
\end{align*}This is equivalent to
\begin{align*}
\langle F_x(z_0) \zeta''(u_0)u,& v\rangle=\\
&-\big[ \langle F_{xx}(z_0)x, \zeta'(u_0)v\rangle +\langle
F_{xu}(z_0)x, v\rangle+\langle
F_{ux}(z_0)u,\zeta'(z_0)v\rangle+\langle F_{uu}(z_0)u,
v\rangle\big].
\end{align*}
It follows that
\begin{equation*}
(\zeta''(u_0)u)^*=-[ F_{xx}(z_0)x\zeta'(u_0)
+F_{xu}(z_0)x+F_{ux}(z_0)u\zeta'(z_0)+F_{uu}(z_0)u]^*
(F_x(z_0)^*)^{-1}.
\end{equation*}
Combining this with formula \eqref{Mulp3},  we have
\begin{align*}
(\zeta''(u_0)u)^*(\lambda &I_x(z_0)+ e^* G_x(z_0))=-(\zeta''(u_0)u)^*\phi\\
&=-[ F_{xx}(z_0)x\zeta'(u_0)
+F_{xu}(z_0)x+F_{ux}(z_0)u\zeta'(z_0)+F_{uu}(z_0)u]^*
(F_x(z_0)^*)^{-1}\phi\\
&=-[ F_{xx}(z_0)x\zeta'(u_0)
+F_{xu}(z_0)x+F_{ux}(z_0)u\zeta'(z_0)+F_{uu}(z_0)u]^* v^*.
\end{align*}
Hence,
\begin{align*}
\langle (\zeta''(u_0)u)^*&(\lambda I_x(z_0)+ e^* G_x(z_0)), u\rangle=\\
&=\langle v^*F_{xx}(z_0)x,\zeta'(u_0)u\rangle +\langle
v^*F_{xu}(z_0)x, u\rangle +\langle F_{ux}(z_0)u,\zeta'(z_0)u\rangle
+\langle v^*F_{uu}(z_0)u, u\rangle\\
&=\langle v^*F_{xx}(z_0)x,x\rangle +\langle v^*F_{xu}(z_0)x,
u\rangle +\langle F_{ux}(z_0)u,x\rangle +\langle v^*F_{uu}(z_0)u,
u\rangle.
\end{align*}Inserting this term into \eqref{Mulp4}, we obtain
\begin{align*}
\langle \lambda \nabla^2 J(u_0)u, u\rangle  &+\langle e^*\nabla^2H(u_0)u, u\rangle=\\
&=\langle \lambda \nabla^2 I(z_0)d, d\rangle+ \langle
v^*\nabla^2F(z_0)d, d\rangle +\langle e^*\nabla^2 G(z_0)d,
d\rangle\geq 0.
\end{align*}The proof is complete.
\end{proof}

\section{Proofs of main results}\label{Section5}
\label{proofs}
\subsection{Proof of Theorem \ref{SOC-necessary-condition}}

For the proof, we first put
\begin{align*}
&X=C([0, 1], \mathbb{R}^n),\ U=L^\infty([0, 1], \mathbb{R}^l),\\
&E_0=C([0, 1], \mathbb{R}^n),\ E= L^\infty([0, 1], \mathbb{R}),
\end{align*} and define the following mappings
\begin{align*}
&F \colon X\times U\to E_0,\ F(x, u):=x-x_0-\int_0^{(\cdot)}\varphi(s,
x(s),
u(s))ds,\\
&G \colon X\times U \to E,\ G(x, u):= g(\cdot, x(\cdot), u(\cdot)).
\end{align*} The problem (MCP) can be formulated in the form of the
problem \eqref{problem3}. Therefore, we can apply Theorem \ref{ThSOC2} for
the (MCP) in order to derive necessary optimality conditions.

\medskip

\noindent {\bf Step 1.} Verification of assumptions $(A1)$--$(A3)$.

 $\bullet$ Verification of $(A1)$. From $(H1)$ we see that
the mapping $I, F$ and $G$ are of class $C^2$ around $\bar z$. Hence,
$(A1)$ is valid. Here $\nabla I_j(\bar z)$, $\nabla^2 I(\bar z)$,
$\nabla F(\bar z)$, $\nabla^2 F(\bar z)$, $\nabla G(\bar z)$ and
$\nabla^2 G(\bar z)$ are defined by:
\begin{align*}
& I_{jx}(\bar z)x=\int_0^1L_{jx}[s]x(s)ds,\  I_{ju}(\bar z)u=\int_0^1L_{ju}[s]u(s)ds,\\
&F_x(\bar z)x= x-\int_0^{(\cdot)}\varphi_x[s]x(s)ds,\ F_u(\bar z)u=
-\int_0^{(\cdot)}\varphi_u[s]u(s)ds),\\
&\nabla G(\bar z)=(G_x(\bar z), G_u(\bar z))=(g_x[\cdot],
g_u[\cdot]),\\
&\langle \nabla^2 I_j(\bar z)z, z\rangle=\int_0^1(\nabla^2
L_j[s]z(s), z(s))ds,\\
&\langle \nabla^2 F(\bar z)z,
z\rangle=-\int_0^{(\cdot)}(\nabla^2\varphi[s]z(s), z(s))ds
\end{align*}for all $z=(x, u)\in Z$, and
\begin{equation}\notag
\nabla^2 G(\bar z)=\left[\begin{array}{cc}
                       g_{xx}[\cdot] & g_{xu}[\cdot] \\
                       g_{ux}[\cdot] &g_{uu}[\cdot]
                     \end{array}
                     \right].
\end{equation}

 $\bullet$ Verification of $(A2)$. Taking any $v\in E_0$,
we consider equation $F_x(\bar z)x=v$. This equation is equivalent to
\begin{equation*}
x=\int_0^{(\cdot)}\varphi_x[s] x(s)ds + v.
\end{equation*}
By assumption $(H1)$, we have $\varphi_x[\cdot]\in L^\infty([0, 1],
\mathbb{R}^n)$. By \cite[Lemma 1, p. 51]{Ioffe}, the equation has a
unique solution $x\in X$. Hence $(A2)$ is valid.

$\bullet$ Verification of $(A3)$. Let $D:=\{(x, u)\in Z \;|\;
F(z)=0\}$. Under assumption  $(A2)$, the mapping $\nabla F(\bar z)\colon X\times U \to E_0$ is surjective. This implies that
\begin{equation*}
T(D; \bar z)=\{(x, u)\in Z \;|\; F_x(\bar z)x + F_u(\bar z)u=0\}.
\end{equation*}
Therefore, assumption $(A3)$ is amount to saying that for
each $v\in E$, there exists $(x, u)\in Z$ satisfying
\begin{align}
&F_x(\bar z)x + F_u(\bar z)u=0,\label{RCon1}\\
&G_x(\bar z)x +G_u(\bar z) u=v.\label{RCon2}
\end{align} We will find $u$ in the form $u=(0, 0, \ldots, u_{i_0}, 0, \ldots, 0)$.
Consider the following equation
\begin{equation*}
F_x(\bar z)x+F_u(\bar z)\frac{v-G_x(\bar z)x}{g_{i_0 u}[\cdot]}=0.
\end{equation*}
This equation is equivalent to
\begin{equation*}
x=\int_0^{(\cdot)}\left(\varphi_x[s] +
\varphi_u[s]\frac{g_x[s]}{g_{i_0 u}[\cdot]}\right)x(s)ds
+\int_0^{(\cdot)}\varphi_u[s]\frac{v}{g_{i_0 u}[s]}ds.
\end{equation*}
By $(H2)$, $\varphi_u[\cdot]\frac{g_x[\cdot]}{g_{i_0 u}[\cdot]}$
and $\varphi_u[\cdot]\frac{v}{g_{i_0 u}[\cdot]}$ belong to
$L^\infty([0, 1],\mathbb{R}^n)$. Thanks to \cite[Lemma 1, p. 51]{Ioffe},
the above equation has a unique solution $x\in X$. Choosing
$u=(0,0, \ldots, u_{i_0}, 0, \ldots, 0)$ with
$$
u_{i_0}= \frac{v-G_x(\bar z)x}{g_{i_0 u}[\cdot]}.
$$
We see that $(x, u)$ satisfies equations \eqref{RCon1}-\eqref{RCon2}. Hence  assumption $(A3)$ is fulfilled.

\medskip

\noindent {\bf Step 2.} Deriving optimality conditions.

Let $\mathcal{L}(z,\lambda, v^*, e^*)=\lambda I(x, u) + v^* F(x, u)
+ e^* G(x, u)$ be the Lagrangian associated with the (MCP). According to
Theorem \ref{ThSOC2}, for each $z=(\tilde x, \tilde u)\in
\mathcal{C}(\bar z)$, there exist multipliers $\lambda\in
\mathbb{R}^m_+$ with $|\lambda|=1$, $v^*\in E_0^*$ and $e^*\in E^*$
such that the following conditions are valid:
\begin{align}
&e^*\in N(Q, G(\bar z)),\label{NC1}\\
&\lambda I_x(\bar z) + v^* F_x(\bar z) + e^* G_x(\bar z)=0,\label{NC2}\\
&\lambda I_u(\bar z) + v^* F_u(\bar z) + e^* G_u(\bar z)=0,\label{NC3}\\
&\langle \lambda \nabla^2 I(\bar z)z, z\rangle + \langle v^*
\nabla^2 F(\bar z)z, z\rangle +\langle e^* \nabla^2 G(\bar z)z,
z\rangle \geq 0.\label{NC4}
\end{align} Here $v^*$ is a signed Radon measure and $e^*$ is a
signed and finite additive measure on $[0, 1]$ which is absolutely
continuous w.r.t the Lebesgue measure $|\cdot|$ on $[0, 1]$. By
Riesz's Representation (see \cite[Chapter 01, p. 19]{Ioffe} and
\cite[Theorem 3.8, p. 73]{Hirsch}), there exists a vector
function of bounded variation $\nu$, which is continuous from the
right and vanish at zero such that
$$
\langle v^*, y\rangle =\int_0^1 y(t)d\nu(t)\ \ \forall y\in E_0,
$$ where $\int_0^1 y(t)d\nu(t)$ is the Riemann-Stieltjes integral.

 Define $\bar p \colon [0, 1]\to \mathbb{R}^n$ by setting
$$
\bar p(t) =\nu((t, 1])=\nu(t)-\nu(1).
$$ Clearly, $\bar p(1)=0$ and the function $\bar p$ is of bounded variation. By
the Fubini Theorem, for each  $x\in C\left([0, 1], \mathbb{R}^n\right)$, we have
\begin{align}
\langle v^* F_x(\bar z), x\rangle&=\left\langle v^*,
x-\int_0^{(\cdot)}\varphi_x[s]x(s)ds\right\rangle\notag\\
&=\int_0^1 x^T(t)d\nu(t)-\int_0^1\int_0^t \varphi_x[s] x(s)ds d\nu(t)\notag\\
&=\int_0^1 x^T(t)d\nu(t) - \int_0^1
\varphi_x[s]x(s)ds \int_s^1d\nu(t)\notag\\
&=\int_0^1 x^T(t)d\nu(t) + \int_0^1 \varphi_x[s]x(s)\bar
p(s)ds.\label{Ac1}
\end{align}
Similarly, for any $u\in L^\infty([0, 1], \mathbb{R}^l)$, we get
\begin{equation}\label{Ac2}
\left\langle v^* F_u(\bar z), u\right\rangle=\int_0^1\int_0^t\varphi_u[s]u(s)ds
d\nu(t)=\int_0^1 \bar p(s)^T \varphi_u[s] u(s) ds
\end{equation} and
\begin{equation}\label{Ac3}
\left\langle v^*\nabla^2F(\bar z)z, z\right\rangle=\int_0^1 \left(\bar
p(t)^T\nabla^2\varphi[s]z(s), z(s)\right)ds.
\end{equation}From \eqref{NC3} and \eqref{Ac2}, we have
\begin{equation}\label{NC5}
\int_0^1\lambda L_u[s]u(s)ds + \int_0^1 \bar p(s)^T \varphi_u[s]
u(s) ds +\langle e^*,G_u(\bar z)u\rangle=0\ \ \forall u\in U.
\end{equation}

Let us claim that $e^*$ can be represented by a density
in $L^1([0, 1], \mathbb{R})$. Indeed, let $\bar d$ be an arbitrary  element of  $T(Q; G(\bar z))$. Then, by assumption $(H2)$, we have
\begin{align*}
|\omega g_u[t]|^2+[(\omega\bar d(t))^+]^2&\geq |\omega g_u[t]|^2
\\
&\geq \omega^2|g_{u_{i_0}}[t]|^2
\\
&\geq \alpha^2\omega^2
\end{align*}
for all $\omega\in\mathbb{R}$ and for a.e. $t\in[0,1]$. Thanks to \cite[Theorem 3.2]{Pales}, there exist measurable mappings $a \colon [0, 1]\times \mathbb{R}\to
\mathbb{R}^l$, $c \colon [0, 1]\times \mathbb{R}\to [0, \infty)$ and a
constant $R>0$ such that
\begin{equation*}
G_u(\bar z)a(t, \omega)=g_u[t]a(t, \omega)=\omega+ c(t, \omega)\bar d(t)
\end{equation*} and
\begin{equation*}
|a(t, \omega)|\leq R|\omega|, |c(t, \omega)|\leq R|\omega|
\end{equation*}
for all $\omega\in \mathbb{R}$. We now take any $v\in E$
and put $u(t)=a(t, v(t))$. Then $u\in U$ and we have
\begin{equation*}
\langle e^*, G_u(\bar z)a(\cdot, v(\cdot))\rangle= \langle e^*,
v\rangle +\langle e^*, c(\cdot, v)\bar d\rangle\leq \langle e^*, v\rangle
\end{equation*} because $e^*\in N(Q, G(\bar z))$ and $c(\cdot, v)\bar d\in
T(Q; G(\bar z))$. Inserting $u(t)=a(t, v(t))$ into \eqref{NC5}, we
get
\begin{equation*}
\int_0^1\lambda L_u[s]a(s, v(s))ds + \int_0^1 \bar p(s)^T
\varphi_u[s] a(s, v(s)) ds +\langle e^*, v\rangle\geq 0\ \ \forall
v\in E.
\end{equation*} This implies that
\begin{equation*}
|\langle e^*, v\rangle|\leq R\int_0^1|\lambda L_u[s]||v(s)|ds +R
\int_0^1|\bar p(s)^T \varphi_u[s]||v(s)| ds.
\end{equation*} From this and \cite[Proposition 5, p. 348]{Ioffe},  there is a function $\theta\in L^1([0, 1],
\mathbb{R})$ such that
\begin{equation}\label{L1-multiplier}
\langle e^*, v\rangle =\int_0^1 \theta(t) v(t)dt\ \ \forall v\in E.
\end{equation}
Therefore the claim is justified.

Based on the representation of $e^*$, \eqref{NC1}, and \eqref{NC5}, we obtain assertions (i) and (iii). Also, from \eqref{L1-multiplier}, \eqref{Ac1} and \eqref{NC2}, we get
\begin{equation*}
\int_0^1\lambda^TL_{x}[s] x(s) ds  +\int_0^1 x^T(t)d\nu(t) +
\int_0^1 \bar p(t)^T\varphi_x[s]x(s)ds + \int_0^1
\theta(s)^Tg_x[s]x(s)ds=0
\end{equation*} for all $x\in X$. This is equivalent to
\begin{equation}\label{KeyEq}
- \int_0^1 x^T(t)d\nu(t) = \int_0^1 \big( \lambda^T L_x[s]+ \bar
p(s)^T\varphi_x[s]+ \theta(s)^Tg_x[s]\big)x(s)ds\ \ \forall x\in X.
\end{equation}  We now fix any vector
$\xi\in\mathbb{R}^n$ and $t\in [0, 1]$. Define $x_t(s)=\xi\chi_{(t,
1]}(s)$, where $\chi_{(t, 1]}(\cdot)$ is the indicator function of
$(t, 1]$. Let us define
$$\vartheta(s)=\lambda^T L_x[s]+ \bar
p(t)^T\varphi_x[s]+ \theta(s)^Tg_x[s].
$$
Then, $\vartheta(\cdot)\in L^1([0, 1], \mathbb{R}^n)$ and so are the
functions $\vartheta(s)$ and $s\vartheta(s)$. By the Lebesgue
Differentiation Theorem (see \cite[Theorem 7.15]{Wheeden}), these
functions have Lebesgue points for  a.e. on $[0, 1]$. Let us denote
by $P$ and $P'$ the sets of Lebesgue points of $\vartheta(s)$ and
$s\vartheta(s)$, respectively. Then $|P\cap P'|=1$  and we have the
following key lemma.

\begin{lemma}\label{Keylemma} For each $t\in P\cap P'$,  the
following equality  is valid:
\begin{equation*}
- \int_0^1 x_t^T(s)d\nu(s) = \int_0^1 \big( \lambda^T L_x[s]+ \bar
p(s)^T\varphi_x[s]+ \theta(s)^Tg_x[s]\big)x_t(s)ds.
\end{equation*}
\end{lemma}
\begin{proof} Note that any function with  bounded
variation as well as any signed Radon measure can be represented as
the difference of two increasing functions, and the difference of
two positive Radon measures, respectively (see \cite[Corollary
2.7]{Wheeden} and \cite[Lemma 13.6]{Taylor}). Therefore, we can
assume that $\nu$ is increasing, right continuous and of bounded
variation.

 For each $\epsilon$ with
$t<\epsilon<1$, we define a function $x_\epsilon$ as follows.
\begin{equation*}
x_\epsilon (s)=
\begin{cases}
\xi\ &\text{if}\ s\in [\epsilon, 1],\\
\frac{\xi(s-t)}{\epsilon-t}\  &\text{if}\ s\in [t, \epsilon],\\
0\  &\text{if}\ s\in [0, t].
\end{cases}
\end{equation*}
Then, $x_\epsilon\in C([0, 1], \mathbb{R}^n)$. By
\eqref{KeyEq}, we have
\begin{equation*}
- \int_0^1 x_\epsilon^T(s)d\nu(s) = \int_0^1 \big( \lambda^T L_x[s]+
\bar p(t)^T\varphi_x[s]+ \theta(s)^Tg_x[s]\big)x_\epsilon(s)ds,
\end{equation*}
or, equivalently,
\begin{equation}\label{Conv0}
-\int_t^\epsilon \frac{\xi(s-t)}{\epsilon-t}
d\nu(s)-\int_\epsilon^1\xi d\nu(s)=\int_t^\epsilon
\frac{\xi(s-t)}{\epsilon-t}\vartheta(s)ds+\int_\epsilon^1\xi\vartheta(s)ds.
\end{equation}
By Mean Value Theorem (see \cite[Theorem 2.27, p. 33]{Wheeden}),
there is a point $t'\in [t, \epsilon]$ such that
$$
\int_t^\epsilon \frac{\xi(s-t)}{\epsilon-t}
d\nu(s)=\frac{\xi(t'-t)}{\epsilon-t}(\nu(\epsilon)-\nu(t)).
$$
Hence
\begin{align*}
 \bigg|\int_t^\epsilon \frac{\xi(s-t)}{\epsilon-t}
d\nu(s)\bigg|&=\frac{|\xi|
|(t'-t)|}{\epsilon-t}(\nu(\epsilon)-\nu(t))
\\
&\leq \frac{|\xi|
|\epsilon-t)|}{\epsilon-t}(\nu(\epsilon)-\nu(t))\\
&\leq |\xi| (\nu(\epsilon)-\nu(t)).
\end{align*}
By letting $\epsilon\to t^+$ and using the right
continuity of $\nu$, we see that
\begin{equation}\label{Conv1}
\bigg|\int_t^\epsilon \frac{\xi(s-t)}{\epsilon-t} d\nu(s)\bigg|\to
0\ \ \text{as}\ \ \epsilon\to t^+.
\end{equation} Also, we have
\begin{equation*}
\bigg|\int_t^1\xi d\nu(s)ds-\int_\epsilon^1\xi
d\nu(s)\bigg|=\bigg|\int_t^\epsilon\xi
d\nu(s)\bigg|\leq|\xi|(\nu(\epsilon)-\nu(t))\to 0 \ \ {\rm as}\ \
\epsilon\to t^+.
\end{equation*}
Consequently,
\begin{equation*}
\int_\epsilon^1\xi d\nu(s)ds\to \int_t^1\xi d\nu(s)\ \ {\rm as}\ \
\epsilon\to t^+.
\end{equation*} For the first term of \eqref{Conv0}, we have from the
Lebesgue Differentiation Theorem (see \cite[Theorem 7.16]{Wheeden}) that
\begin{align*}
\bigg| \int_t^\epsilon
\frac{\xi(s-t)}{\epsilon-t}\vartheta(s)ds\bigg|&\leq
|\xi|\frac{1}{\epsilon-t}\int_t^\epsilon |(s-t)\vartheta(s)| ds
\\
&\leq |\xi|\frac{1}{\epsilon-t}\int_t^\epsilon
|s\vartheta(s)-t\vartheta(t)|ds+|\xi|\frac{t}{\epsilon-t}
\int_t^\epsilon |\vartheta(t)-\vartheta(s)|ds\to 0
\end{align*}
as $\epsilon\to t^+$. Hence
\begin{equation*}
\bigg| \int_t^\epsilon
\frac{\xi(s-t)}{\epsilon-t}\vartheta(s)ds\bigg|\to 0\ \ {\rm as}\ \
\epsilon\to t^+.
\end{equation*} The convergence
\begin{equation}\label{Conv4}
\int_\epsilon^1\xi^T\vartheta(s)ds\to \int_t^1\xi^T\vartheta(s)ds\ \
{\rm as}\ \ \epsilon\to t^+
\end{equation} is straightforward.  Passing the limit both sides of
\eqref{Conv0} and using \eqref{Conv1}--\eqref{Conv4}, we obtain
\begin{equation*}
-\int_t^1\xi^T d\nu(s)=\int_t^1 \xi^T\vartheta(s) ds.
\end{equation*}
The proof of the lemma is complete.
\end{proof}

From Lemma \ref{Keylemma}, we have for a.e. $t\in [0, 1]$ that
\begin{equation*}
 -\int_t^1 \xi^T d\nu(s)= \int_t^1 \big( \lambda^T
L_x[t]+ \bar p(t)^T\varphi_x[t]+ \theta(t)g_x[t]\big)\xi ds
\end{equation*}
or, equivalently,
\begin{equation*}
\xi^T \bar p(t)=\int_t^1 \xi^T\big( \lambda^T L_x[t]+ \bar
p(t)^T\varphi_x[t]+ \theta(t)g_x[t]\big)ds.
\end{equation*} Since $\xi$ is arbitrary,  we obtain
\begin{align*}
&\dot{\bar p}(t)=-\lambda^T L_x[t]- \bar p(t)^T\varphi_x[t]-
\theta(t)g_x[t]\quad {\rm a.e.}\quad t\in [0, 1],\\
& \bar p(1)=0,
\end{align*} which is assertion (ii) of Theorem \ref{SOC-necessary-condition}. Finally, from \eqref{Ac3} and \eqref{NC4}, we have
\begin{align*}
\int_0^1\left(\sum_{j=1}^m\lambda_j\nabla^2 L_j[t]z(t),z(t)\right)dt&+ \int_0^1
\left(\bar p(t)^T\nabla^2\varphi[t]z(t), z(t)\right)dt
\\
&+\int_0^1\left(\theta(t)\nabla^2 g[t]z(t), z(t)\right)dt\geq 0,
\end{align*}
which is assertion (iv) of the theorem. The proof of Theorem \ref{SOC-necessary-condition} is complete.

\subsection{Proof of Theorem \ref{SOC-sufficient-condition}}

In this proof, we will use the Sobolev space $W^{1,2}([0, 1],
\mathbb{R}^n)$ which consists of absolutely continuous functions $x$
with $\dot x\in L^2([0, 1], \mathbb{R}^n)$.

Let us define the Lagrangian $\mathcal{L}(z, \lambda, \bar p,
\theta)$ associated with the (MCP) by setting
$$
\mathcal{L}(z, \lambda, \bar p, \theta):=\lambda^T I(z) + \bar p^T
F(z) +\theta G(z),
$$ where
\begin{align*}
&\lambda^T I(z)=\int_0^1\lambda^T L(s, x(s), u(s))ds,\\
& \bar p^T F(z)=\int_0^1 \dot{\bar p}(s)x(s)ds +\int_0^1\bar p(s)
\varphi(s, x(s), u(s))ds,\\
&\theta G(z)=\int_0^1 \theta(s) g(s, x(s), u(s))ds.
\end{align*}
Then, from conditions (ii) and (iii) of Theorem \ref{SOC-necessary-condition}, we can show that
\begin{equation}\label{FCon}
\nabla_z \mathcal{L}(\bar z, \lambda, \bar p, \theta)=0.
\end{equation}

We now return to the proof of the theorem. Suppose the the theorem was
false. Then, we could find sequences $\{(x_k, u_k)\}\subset \Phi$ and
$\{c_k\}\subset \mathrm{int}\,\mathbb{R}^m_+$ such that $(x_k, u_k) \to (\bar x,
\bar u)$,  $c_k\to 0$ and
\begin{equation}\label{Iq1}
I(x_k, u_k)- I(\bar x, \bar u) -c_k\|u_k-\bar u\|_2^2\in
-\mathbb{R}^m_+\setminus\{0\}.
\end{equation}
Clearly, $(x_k, u_k)\neq (\bar x, \bar u)$ for all $k\in\mathbb{N}$. By replacing the sequence  $\{(x_k, u_k)\}$ by a subsequence  we may assume that $u_k=\bar u$ or $u_k\neq \bar u$ for all $k\in\mathbb{ N}$. 
If $u_k=\bar u$ for all $k\in\mathbb{N}$, then we have
$$
x_k(t)=x_0+\int_0^t\varphi(s, x_k(s), \bar u(s))ds
$$ and
$$
\bar x(t)=x_0+\int_0^t\varphi(s, \bar x(s), \bar u(s))ds.
$$ 
Hence,
$$
x_k(t)-\bar x(t)=\int_0^t(\varphi(s, x_k(s), \bar u(s))-\varphi(s,
\bar x(s), \bar u(s)))ds.
$$ From this and $(H1)$, there exist numbers $M>0$ and $k_{\varphi M}>0$ such
that for $k$ large enough, we have
$$
|\varphi(s, x_k(s), \bar u(s))-\varphi(s, \bar x(s), \bar
u(s))|\leq k_{\varphi M}|x_k(t)-\bar x(t)|.
$$ Hence,
$$
|x_k(t)-\bar x(t)|\leq \int_0^t k_{\varphi M}|x_k(s)-\bar x(s)|ds.
$$ Using the Gronwall Inequality (see \cite[18.1.i, p. 503]{Lamberto}), we get $x_k=\bar x$, a contradiction. Therefore, we
have that $u_k\neq \bar u$ for all $k\in \mathbb{N}$.

 Define $t_k=\|u_k- \bar u\|_2$, $\hat x_k= \frac{x_k-\bar x}{t_k}$ and $\hat u_k= \frac{u_k-\bar u}{t_k}$.
Then $t_k\to 0^+$ and $\|\hat u_k \|_2=1$. Since $L^2([0, 1],
\mathbb{R}^l)$ is reflexive, we may assume that $\hat u_k
\rightharpoonup \hat u$. From the above, we have
\begin{equation}\label{KeyIq2}
\lambda^T I(z_k)-\lambda^T I(\bar z)\leq t_k^2\lambda^T c_k\leq
t_k^2 |\lambda||c_k|\leq o(t_k^2).
\end{equation} We claim that $\hat x_k$ converges uniformly to some
$\hat x$ in $C([0,1], \mathbb{R}^n)$. In fact, since $(x_k,
u_k)\in\Phi$, we have
\begin{equation*}
x_k(t)= x_0+\int_0^t \varphi(s, x_k(s), u_k(s))ds.
\end{equation*} Since $x_k=\bar x + t_k \hat x_k$, we have
\begin{equation}\label{DE1}
t_k\hat x_k(t)=\int_0^t (\varphi(s, x_k(s), u_k(s))-\varphi(s, \bar
x(s), \bar u(s))ds.
\end{equation} Since $x_k\to \bar x$ uniformly and $u_k\to \bar u$ in
$L^\infty([0, 1], \mathbb{R}^l)$, there exists a constant
$\varrho>0$ such that $\|x_k\|_0\leq\varrho,
\|u_k\|_\infty\leq\varrho$. By assumption $(H1)$, there exists
$k_{\varphi, \varrho}>0$ such that
$$
|\varphi(s, x_k(s), u_k(s))-\varphi(s, \bar x(s), \bar u(s))|\leq
k_{\varphi, \varrho}(|x_k(s)-\bar x(s)| +|u_k(s)-\bar u(s)|)
$$ for a.e. $s\in[0, 1]$. Hence we have from \eqref{DE1} that
\begin{equation*}
|\hat x_k(t)|\leq \int_0^t k_{\varphi, \varrho}(|\hat x_k(s)|+ |\hat
u_k(s)|)ds
\end{equation*} and
\begin{equation}\label{DEInq1}
|\dot{\hat x}_k(t)|\leq  k_{\varphi, \varrho}(|\hat x_k(t)|+ |\hat
u_k(t)|).
\end{equation} It follows that
\begin{align*}
|\hat x_k(t)|&\leq \int_0^t k_{\varphi, \varrho}|\hat x_k(s)|ds+
\int_0^1 k_{\varphi, \varrho}|\hat u_k(s)|ds\\
&\leq \int_0^tk_{\varphi, \varrho}|\hat x_k(s)|ds+ k_{\varphi,
    \varrho}\left(\int_0^1 |\hat u_k(s)|^2 ds\right)^{1/2}\\
&\leq \int_0^t k_{\varphi, \varrho}|\hat x_k(s)|ds+ k_{\varphi,
    \varrho}, \quad \quad (\|\hat u_k\|_2=1).
\end{align*}Using the Gronwall Inequality, we have
\begin{equation*}
|\hat x_k(t)|\leq k_{\varphi, \varrho}\exp(k_{\varphi, \varrho}).
\end{equation*}  From this and
\eqref{DEInq1}, we  see that
\begin{equation*}
|\dot{\hat x}_k(t)|^2\leq  2k^2_{\varphi, \varrho}\left(|\hat x_k(t)|^2+
|\hat u_k(t)|^2\right)\leq  2k^2_{\varphi, \varrho}(k^2_{\varphi,
    \varrho}\exp(2k_{\varphi, \varrho}) + |\hat u_k|^2).
\end{equation*}
Hence,
$$
\int_0^1 |\dot{\hat x}_k(t)|^2 dt\leq 2k^2_{\varphi,
    \varrho}(k^2_{\varphi, \varrho}\exp(2k_{\varphi, \varrho})+1).
$$ Consequently, $\{\hat x_k\}$ is bounded in $W^{1,2}([0, 1],
\mathbb{R}^n)$. By passing subsequence, we can assume that $\hat
x_k\rightharpoonup \hat x$ weakly in $W^{1,2}([0, 1],
\mathbb{R}^n)$. Thanks to \cite[Theorem 8.8]{Brezis1}, the embedding
$W^{1,2}([0, 1], \mathbb{R}^n)\hookrightarrow C([0, 1],
\mathbb{R}^n)$ is compact. Hence we have that $\hat x_k\to\hat x$
uniformly on $[0, 1]$. The claim is justified. The remains of the
proof is divided into  some steps.

\medskip

\noindent {\bf Step 1}. Showing that $(\hat x, \hat
u)\in\mathcal{C}'(\bar z)$.

By a Taylor expansion, we have from \eqref{Iq1} that
\begin{equation}\label{CriticalCon1}
I_x(\bar z)\hat x_k + I_u(\bar z)\hat u_k
+\frac{o(t_k)}{t_k}\in-\mathbb{R}_+^m.
\end{equation}Note that $L_{ju}[\cdot]\in L^\infty([0, 1], \mathbb{R}^l)$  and
$I_{ju}(\bar z) \colon L^2([0, 1], \mathbb{R}^l)\to\mathbb{R}$ is a continuous linear mapping,
where
$$
\langle I_{ju}(\bar z), u\rangle:=\int_0^1 L_{ju}[s] u(s) ds\ \quad
\forall u\in L^2([0, 1], \mathbb{R}^l).
$$ By \cite[Theorem 3.10]{Brezis1}, $I_{ju}(\bar z)$ is weakly
continuous on $L^2([0, 1], \mathbb{R}^l)$. By letting $k\to\infty$
in \eqref{CriticalCon1}, we get
\begin{equation}\label{c1}
I_x(\bar z) \hat x +I_u(\bar z)\hat u\in-\mathbb{R}^m_+.
\end{equation} Since $F(\bar z)=0, F(x_k, u_k)=0$ and by a Taylor
expansion, we have
$$
F_x(\bar z)\hat x_k + F_u(\bar z)\hat u_k +\frac{o(t_k)}{t_k}=0.
$$ By the same arguments as the above and letting $k\to\infty$, we
obtain
\begin{equation}\label{c2}
F_x(\bar z)\hat x + F_u(\bar z)\hat u=0.
\end{equation} Since $G(x_k, u_k)-G(\bar x, \bar u)\in Q-G(\bar x, \bar
u)$ and by a Taylor expansion, we have
$$
G_x(\bar z)\hat x_k + G_u(\bar z)\hat u_k +\frac{o(t_k)}{t_k}\in
{\rm cone}(Q-G(\bar x, \bar u))\subset T(Q; G(\bar x, \bar u)),
$$ where $T(Q; G(\bar x, \bar
u))$ is the tangent cone to $Q$ at $G(\bar x, \bar u)$ in
$L^\infty([0, 1], \mathbb{R})$. It is easily seen that
\begin{align*}
T(Q; G(\bar x, \bar u))&\subseteq\left\{ v\in L^\infty([0, 1],
\mathbb{R}) \;|\; v(t)\in T((-\infty, 0]; g[t])\quad {\rm a.e.}\right\}\\
&\subseteq\left\{ v\in L^2([0, 1], \mathbb{R}) \;|\; v(t)\in T((-\infty, 0];
g[t])\quad {\rm a.e.}\right\}.
\end{align*}
Hence,
\begin{equation}\label{CriticalCon2}
G_x(\bar z)\hat x_k + G_u(\bar z)\hat u_k +\frac{o(t_k)}{t_k}\in\left\{
v\in L^2([0, 1], \mathbb{R}) \;|\; v(t)\in T((-\infty, 0]; g[t])\quad {\rm
    a.e.}\right\}.
\end{equation} Note that
$$
\left\{ v\in L^2([0, 1], \mathbb{R}) \;|\; v(t)\in T((-\infty, 0]; g[t])\ {\rm
    a.e.}\right\}=T_{L^2}(Q; G(\bar x, \bar u)),
$$ where $T_{L^2}(Q; G(\bar x, \bar
u))$ is the tangent cone to the set $Q$ at $G(\bar x, \bar u)$ in $L^2([0,
1], \mathbb{R})$. Since  $T_{L^2}(Q; G(\bar x, \bar u))$ is a closed
convex set in $L^2([0, 1], \mathbb{R})$, it is also a weakly closed
set in $L^2([0, 1], \mathbb{R})$. Since
$$
G_u(\bar z) \colon L^2([0, 1], \mathbb{R}^l)\to L^2([0, 1], \mathbb{R})
$$ is a continuous linear mapping, \cite[Theorem 3.10]{Brezis1}
implies that it is continuous from  weakly topology of $L^2([0, 1],
\mathbb{R}^l)$ to weakly topology of $ L^2([0, 1], \mathbb{R})$. By
passing the limit in \eqref{CriticalCon2} when $k\to\infty$, we
obtain
\begin{equation*}
G_x(\bar z)\hat x + G_u(\bar u) \hat u\in \left\{ v\in L^2([0, 1],
\mathbb{R}) \;|\; v(t)\in T((-\infty, 0]; g[t])\ {\rm a.e.}\right\}.
\end{equation*} Combining this with \eqref{c1} and \eqref{c2}, we get
$(\hat x, \hat u)\in\mathcal{C}'(\bar z)$.

\medskip

\noindent {\bf Step 2}. Showing that $(\hat x, \hat u)=0$.

By a second-order Taylor expansion for $\mathcal{L}$ and
\eqref{FCon}, we get
\begin{equation*}
\mathcal{L}(z_k, \lambda ,\bar p, \theta)-\mathcal{L}(\bar
z,\lambda,\bar p, \theta)= \frac{t_k^2}{2}\nabla^2_{zz}\mathcal{L}(
\bar z, \lambda, \bar p, \theta)(\hat z_k,\hat z_k) +o(t_k^2),\quad
(\hat z_k=(\hat x_k, \hat u_k)).
\end{equation*}
On the other hand from \eqref{KeyIq2}, we have
\begin{equation*}
\mathcal{L}(z_k, \lambda,\bar p, \theta)-\mathcal{L}(\bar z,
\lambda,\bar p, \theta)= \lambda^T(I(z_k)-I(\bar z)) +\langle
\theta,G(z_k)-G(\bar z)\rangle \leq o(t_k^2).
\end{equation*} Here we used the fact that $\theta \in N(Q, G(\bar z))$ and $F(z_k)=F(\bar z)=0$. Therefore, we have
$$
\frac{t_k^2}{2}\nabla^2_{zz}\mathcal{L}(\bar z, \lambda, \bar p,
\theta)(\hat z_k, \hat z_k) +o(t_k^2)\leq o(t_k^2),
$$
or, equivalently,
\begin{equation}\label{iqL1}
\nabla^2_{zz}\mathcal{L}(\bar z,\lambda, \bar p, \theta)(\hat z_k,
\hat z_k) \leq \frac{o(t_k^2)}{t_k^2}.
\end{equation} By letting $k\to\infty$,   we obtain
\begin{equation*}
\nabla^2_{zz}\mathcal{L}(\bar z, \lambda, \bar p, \theta)(\hat z,
\hat z) \leq 0.
\end{equation*}
By a simple calculation, we have
\begin{align*}
\int_0^1\left(\lambda^T\nabla^2 L[t]\hat z(t),\hat z(t)\right)dt+ \int_0^1
\left(\bar p(t)^T\nabla^2\varphi[t]\hat z(t), \hat z(t)\right)dt
&+\int_0^1\left(\theta(t)\nabla^2 g[t]\hat z(t), \hat z(t)\right)dt\\
&=\nabla^2_{zz}\mathcal{L}(\bar z, \lambda, \bar p, \theta)(\hat z,
\hat z)\leq 0.
\end{align*} Combining this with \eqref{StrSOC}, we must have $\hat z=0$.

\medskip

\noindent {\bf Step 3.} Showing a contradiction.

From \eqref{StrSOC1} and \eqref{iqL1}, we have
\begin{align*}
\frac{o(t_k^2)}{t_k^2}\geq & \nabla_{zz}^2\mathcal{L}(\bar z,\lambda,
\bar p, \theta)(\hat z_k, \hat z_k)\\
=& \int_0^1\left(\lambda^T\nabla^2 L[t]\hat z_k(t),\hat
z_k(t)\right)dt+ \int_0^1 \left(\bar p(t)^T\nabla^2\varphi[t]\hat
z_k(t), \hat z_k(t)\right)dt
\\
&+\int_0^1\left(\theta(t)\nabla^2 g[t]\hat z_k(t), \hat z_k(t)\right)dt\\
=& \int_0^1\lambda^T L_{uu}[t]\hat u^2_k(t)dt+ 2\int_0^1\lambda^T
L_{xu}[t]\hat x_k(t)\hat u_k(t)dt +\int_0^1 \lambda^T
L_{xx}[t]\hat x^2_k(t)dt+\\
&+\int_0^1 \left(\bar p(t)^T\nabla^2\varphi[t]\hat z_k(t), \hat z_k(t)\right)dt
+\int_0^1\left(\theta(t)\nabla^2 g[t]\hat z_k(t), \hat z_k(t)\right)dt\\
\geq& \gamma_0+ 2\int_0^1\lambda^T L_{xu}[t]\hat x_k(t)\hat u_k(t)dt
+\int_0^1 \lambda^T
L_{xx}[t]\hat x^2_k(t)dt+\\
&+\int_0^1 \left(\bar p(t)^T\nabla^2\varphi[t]\hat z_k(t), \hat z_k(t)\right)dt
+\int_0^1\left(\theta(t)\nabla^2 g[t]\hat z_k(t), \hat z_k(t)\right)dt.
\end{align*}  By letting $k\to \infty$ and using the fact $\hat z=0$,  we
obtain $0\geq \gamma_0$, which is impossible. The proof of Theorem
\ref{SOC-sufficient-condition} is complete.

\section{Examples}
In this section, we give some examples to illustrate the main results.
 The first example shows us how to use Theorem \ref{SOC-necessary-condition}
 and Theorem \ref{SOC-sufficient-condition} to obtain solutions of the (MCP).
 The second one  indicates the important role of the second-order necessary optimality conditions in
 checking optimal solutions.
\label{illustrate-example}
\begin{example}\label{example-1}{\rm Consider the problem (MCP), where
\begin{align*}
L(t, x(t), u(t))&=(x_1^2(t)+u_1^2(t), x_2^2(t)+u_2^2(t)),
\\
\varphi(t, x(t), u(t))&=(u_1(t), u_2(t)),
\\
x_0&=(0,0),
\\
g(t, x(t), u(t))&=x_1(t)+x_2(t)-u_1(t)-u_2(t)
\end{align*}
for all $x(t)=(x_1(t), x_2(t))$, $u(t)=(u_1(t), u_2(t))$ and $t\in [0, 1]$. Then, the feasible solution set of the (MCP) is
\begin{align*}
\Phi=\bigg\{(x, u)\in C([0, 1], \mathbb{R}^2)\times L^{\infty}([0, 1], \mathbb{R}^2&)\,|\, x(t)=\int_{0}^{t}u(s)ds,
\\
&x_1(t)+x_2(t)-u_1(t)-u_2(t)\leq 0 \ \ \text{a.e.}\ \ t\in [0,1]\bigg\}.
\end{align*}
It is easy to see that conditions $(H1)$ and $(H2)$  are valid. We
will use conditions (i)--(iii) of Theorem
\ref{SOC-necessary-condition} to find out KKT points of the (MCP)  which
are good candidates for optimal solutions. Assume that  $\bar
z=(\bar x, \bar u)$ is a feasible solution of the (MCP) and satisfies
conditions (i)--(iii) of Theorem \ref{SOC-necessary-condition} with
respect to $(\lambda, \bar p, \theta)$.  By simple computations, we
have
\begin{align*}
L_{1x}[t]&=(2\bar x_1,0)^T, L_{2x}[t]=(0, 2\bar x_2)^T,  \varphi_{1x}[t]=\varphi_{2x}[t]=(0,0)^T, g_x[t]=(1, 1)^T,
\\
L_{1u}[t]&=(2\bar u_1,0)^T, L_{2u}[t]=(0, 2\bar u_2)^T, \varphi_{1u}[t]=(1,0)^T, \varphi_{2u}[t]=(0,1)^T, g_u[t]=(-1, -1)^T.
\end{align*}
From conditions (ii) and (iii), we get
\begin{equation}\label{equ-exam-1}
\begin{cases}
\dot{\bar p}_1=-2\lambda_1\bar x_1-\theta,
\\
\dot{\bar p}_2=-2\lambda_2\bar x_2-\theta,
\\
\bar{p}_1(1)=\bar{p}_2(1)=0,
\end{cases}
\end{equation}
and
\begin{equation}\label{equ-exam-2}
\begin{cases}
2\lambda_1\bar u_1+\bar p_1-\theta=0,
\\
2\lambda_2\bar u_2+\bar p_2-\theta=0.
\end{cases}
\end{equation}
Combining  \eqref{equ-exam-1} and \eqref{equ-exam-2} yields
\begin{equation}\label{equ-exam-3}
\begin{cases}
\dot{\bar p}_1=-2\lambda_1\bar x_1-2\lambda_1\bar u_1-\bar{p}_1,
\\
\dot{\bar p}_2=-2\lambda_2\bar x_2-2\lambda_2\bar u_2-\bar{p}_2.
\end{cases}
\end{equation}
By condition \eqref{P2}, one has
\begin{equation}\label{equ-exam-4}
\begin{cases}
\dot{\bar x}_1=\bar u_1,
\\
\dot{\bar x}_2=\bar u_2.
\end{cases}
\end{equation}
Then inserting these equations into \eqref{equ-exam-3}, we obtain
\begin{equation*}
\begin{cases}
\dot{\bar p}_1=-2\lambda_1\bar x_1-2\lambda_1\dot{\bar x}_1-\bar{p}_1,
\\
\dot{\bar p}_2=-2\lambda_2\bar x_2-2\lambda_2\dot{\bar x}_2-\bar{p}_2,
\end{cases}
\Leftrightarrow
\begin{cases}
\dot{\bar p}_1+2\lambda_1\dot{\bar x}_1=-\bar{p}_1-2\lambda_1\bar x_1,
\\
\dot{\bar p}_2+2\lambda_2\dot{\bar x}_2=-\bar{p}_2-2\lambda_2\bar x_2.
\end{cases}
\end{equation*}
This implies that
\begin{equation}\label{equ-exam-8}
\begin{cases}
\bar{p}_1+2\lambda_1\bar x_1=c_1 \exp{(-t)},
\\
\bar{p}_2+2\lambda_2\bar x_2=c_2 \exp{(-t)},
\end{cases}
\end{equation}
where $c_1, c_2\in\mathbb{R}$ are constants. Hence,
\begin{equation}\label{equ-exam-5}
\bar{p}_1-\bar{p}_2+2\lambda_1\bar x_1-2\lambda_2\bar x_2=c_3\exp{(-t)},
\end{equation}
where $c_3:=c_1-c_2$. From \eqref{equ-exam-2} and
\eqref{equ-exam-4}, we have
\begin{equation*}
\bar{p}_1-\bar{p}_2=-2\lambda_1\dot{\bar x}_1+2\lambda_2\dot{\bar x}_2.
\end{equation*}
Inserting this equation into \eqref{equ-exam-5}, we get
\begin{equation} \label{equ-exam-6}
-2\lambda_1\dot{\bar x}_1+2\lambda_2\dot{\bar x}_2+2\lambda_1\bar x_1-2\lambda_2\bar x_2=c_3\exp{(-t)},
\end{equation}
Let $\alpha=2\lambda_1\bar x_1-2\lambda_2\bar x_2$. Then equation \eqref{equ-exam-6} becomes
\begin{equation*}
-\dot{\alpha}+\alpha=c_3\exp(-t).
\end{equation*}
From this equation and $\alpha(0)=0$, it is easily seen that
$$\alpha(t)=\frac{1}{2}c_3\exp(-t) -\frac{1}{2}c_3\exp(t)\ \ \forall t\in [0, 1].$$
This and \eqref{equ-exam-4} imply that
\begin{equation}\label{equ-exam-7}
\begin{cases}
2\lambda_1\bar x_1-2\lambda_2\bar x_2=\frac{1}{2}c_3\exp(-t) -\frac{1}{2}c_3\exp(t),
\\
2\lambda_1\bar u_1-2\lambda_2\bar u_2=-\frac{1}{2}c_3\exp(-t) -\frac{1}{2}c_3\exp(t).
\end{cases}
\end{equation}
We see that, for each $\lambda\in\mathbb{R}^2_+\setminus\{0\}$,
every solution $(\bar x, \bar u)$ of  system \eqref{equ-exam-7} is a
KKT point of the (MCP). To illustrate Theorem
\ref{SOC-necessary-condition}, let us verify condition (iv) at a
solution of the system \eqref{equ-exam-7}. Let $\bar x=(0, 0), \bar
u=(0, 0)^T$.  By \eqref{equ-exam-8} and $\bar{p}(1)=(0,0)$, we
have $c_1=0$, $c_2=0$, and $\bar p(t)=(0,0)$ for all $t\in [0, 1]$.
Consequently, $c_3=0$ and $\theta(t)=0$ for all $t\in [0, 1]$. Thus,
$(\bar x, \bar u)$ is a solution of the system \eqref{equ-exam-7}
for every $\lambda\in\mathbb{R}^2_+\setminus\{0\}$. By simple
calculation, we get
\begin{equation*}
\int_0^1\left(\sum_{j=1}^m\lambda_j\nabla^2 L_j[t]z(t),z(t)\right)dt=2\lambda_1\int_{0}^{1}\left(x_1^2(t)+u_1^2(t)\right)dt+2\lambda_2\int_{0}^{1}\left(x_2^2(t)+u_2^2(t)\right)dt\geq 0
\end{equation*}
for all $z=(x, u)\in X\times U$ and $\lambda\in\mathbb{R}^2_+\setminus\{0\}$. Hence, condition (iv) is satisfied.

We now use Theorem \ref{SOC-sufficient-condition} to show that
$(\bar x, \bar u)$ is a locally strong Pareto solution of the (MCP). Let
$\lambda=(\frac{1}{2},\frac{1}{2})$, $\bar p=(0,0)$ and $\theta=0$.
Then, we have
\begin{equation*}
\int_0^1\left(\sum_{j=1}^m\lambda_j\nabla^2 L_j[t]z(t),z(t)\right)dt=\int_{0}^{1}\left(x_1^2(t)+u_1^2(t)\right)dt+\int_{0}^{1}\left(x_2^2(t)+u_2^2(t)\right)dt>0
\end{equation*}
for all $z=(x, u)\in X\times U\setminus\{(0,0)\}$. Hence condition \eqref{StrSOC} is satisfied.  Furthermore, we see that
\begin{equation*}
\lambda^T L_{uu}[t](\xi,\xi)=|\xi|^2
\end{equation*}
for all $\xi\in \mathbb{R}^2$. This implies that condition
\eqref{StrSOC1} holds at $\bar z$ with respect to $\gamma_0=1$.
Thanks to Theorem \ref{SOC-sufficient-condition}, $\bar z$ is a
locally strong Pareto solution of the (MCP).
}
\end{example}
\begin{example}{\rm  Let $\varphi$ and $g$ be as in Example \ref{example-1} and $L$ be defined by
$$
L(t, x(t), u(t))=(x_1^2(t)-u_1^2(t), x_2^2(t)-u_2^2(t))
$$ for all $x(t)=(x_1(t), x_2(t))$, $u(t)=(u_1(t), u_2(t))$ and $t\in [0, 1]$.
Clearly, $(\bar x, \bar u)=((0,0),(0,0))$ is a feasible point of the
(MCP). We claim that $(\bar x, \bar u)$ is not a locally weak Pareto
solution of the (MCP). Indeed, if otherwise, then due to Theorem
\ref{SOC-necessary-condition}, for each critical direction
$z\in\mathcal{C}(\bar z)$, there exist multipliers $\lambda$, $\bar
p$  and  $\theta$ such that conditions (i)--(iv) are fulfilled. Let
$\tilde x(t)=(t, t)^T$ and $\tilde u(t)=(1,1)^T$ for all $t\in [0,
1]$. It is easy to check that $z:=(\tilde x, \tilde u)$ is a
critical direction of the (MCP) at $\bar z$.  By conditions (ii) and
(iii) of Theorem 2.1, we get
\begin{equation*}
\begin{cases}
\dot{\bar p}=-\theta(1,1)^T,
\\
p(1)=0,
\\
\bar p+\theta(-1, -1)^T=0,
\end{cases}
\Leftrightarrow
\begin{cases}
\dot{\bar p}_1=\dot{\bar p}_2=-\theta,
\\
p(1)=0,
\\
\bar{p}_1=\bar{p}_2=\theta.
\end{cases}
\end{equation*}
This implies that $\bar{p}_1(t)=\bar{p}_2(t)=\theta(t)=0$ for all $t\in [0,1]$.
Since $\lambda\in\mathbb{R}^2_+\setminus\{0\}$, $\bar p=(0, 0)^T$ and $\theta=0$, we have
\begin{align*}
&\int_{0}^{1}\left(\lambda_1\nabla^2L_1[t]z(t)+\lambda_2\nabla^2L_2[t]z(t),
z(t)\right)dt+ \int_0^1 \left(\bar p(t)^T\nabla^2\varphi[t]z(t),
z(t)\right)dt
\\
&+\int_0^1\left(\theta(t)\nabla^2 g[t]z(t), z(t)\right)dt
\\
=&\int_{0}^{1}\left(\lambda_1\nabla^2L_1[t]z(t)+\lambda_2\nabla^2L_2[t]z(t), z(t)\right)dt
\\
=&\int_{0}^{1}(\lambda_1+\lambda_2)(2t^2-2)dt=-\frac{4}{3}(\lambda_1+\lambda_2)<
0,
\end{align*}
which does not satisfy  condition (iv) of Theorem 2.1. Hence, $(\bar
x, \bar u)$ is not a locally weak Pareto solution of the (MCP). }
\end{example}

\section*{Acknowledgments}
The research of B.T. Kien was  partially supported by the
Vietnam-Russia Project {\it ``Development of the effective
algorithms for mathematical models of power plants equipment''}. J.-C. Yao was partially supported by the Grant MOST 105-2221-E-039-009-MY3.

\end{document}